\documentclass[10pt]{llncs}

\usepackage{mathtools}
\usepackage{amssymb}
\usepackage{algorithm,algorithmicx} 
\usepackage[noend]{algpseudocode}
\usepackage{tikz}
\usetikzlibrary{arrows,automata,positioning, shapes.symbols,shapes.callouts,patterns,calc,shapes.geometric,fit,arrows.meta}
\usepackage{stmaryrd}
\usepackage[pdfpagelabels=true]{hyperref}
\usepackage{cleveref}
\hypersetup{
	linktoc = page,
	pdfpagemode = UseNone,
	colorlinks,
	linkcolor={red!50!black},
	citecolor={blue!50!black},
	urlcolor={blue!80!black}
}

\newcommand{\dtcs}[4]
{[\begin{smallmatrix}
  #1  & #2 \\
  #3 & #4
\end{smallmatrix}]}
\newcommand{\den}[1]{$\phantom{|^{|^|}}$#1$\phantom{|^{|^|}}$}
\newcommand{\ssz}[1]
{\mbox{ss}{\mathbb{Z}}_{#1}}

\newcommand{\tcs}[3]
{[\begin{smallmatrix}
    & #1\\
  #2 & #3
\end{smallmatrix}]}

\newcommand{\dec}[4]{\mbox{\textsf{dec}}\llbracket \begin{smallmatrix}
 1 & 3  & 5 & 7 & 9 \\
 * &  #1 & #2 & #3 & #4
\end{smallmatrix}\rrbracket}

\newcommand{\decall}[8]{\mbox{\textsf{dec}}\llbracket \begin{smallmatrix}
 1 & 2 & 3 & 4 & 5 & 6 & 7 & 8 & 9 \\
 * &  #1 & #2 & #3 & #4 & #5 & #6 & #7 & #8
\end{smallmatrix}\rrbracket}

\newcommand{\ord}{\mbox{\textsf{ord}}}
\newcommand{\dnv}{\mbox{\textsc{dnv}}}
\newcommand{\anad}{\mbox{\textsc{anad}}}
\newcommand{\aad}{\mbox{\textsc{aad}}}

\tikzset{
  node/.style = {shape=rectangle, rounded corners,
                     draw, align=center},
}

\begin{document}

\title{The Multiplicative Persistence Conjecture Is True for Odd Targets}

\author{\'Eric Brier \and Christophe Clavier\inst{2} \and Linda Gutsche\inst{1} \and David Naccache\inst{1}} 

\institute{ENS, CNRS, PSL Research University\\
  Département d'informatique, \'Ecole normale supérieure, Paris, France\\
 \and
Université de Limoges, XLIM-MATHIS, Limoges, France\\
\email{linda.gutsche@ens.psl.eu}, \email{david.naccache@ens.fr}\\
\email{eric.brier@polytechnique.org}\\
\email{christophe.clavier@unilim.fr}
}

\maketitle

\begin{abstract}
In 1973, Neil Sloane published a very short paper introducing an intriguing problem: Pick a decimal integer $n$ and multiply all its digits by each other. Repeat the process until a single digit $\Delta(n)$ is obtained. $\Delta(n)$ is called the \textsl{multiplicative digital root of $n$} or \textsl{the target of $n$}. The number of steps $\Xi(n)$ needed to reach $\Delta(n)$, called \textsl{the multiplicative persistence of $n$} or \textsl{the height of $n$} is conjectured to always be at most $11$.\smallskip

Like many other very simple to state number-theoretic conjectures, the multiplicative persistence mystery resisted numerous explanation attempts.\smallskip

This paper proves that the conjecture holds for all odd target values:
\begin{itemize}
    \item If $\Delta(n)\in\{1,3,7,9\}$, then $\Xi(n) \leq 1$
    \item If $\Delta(n)=5$, then $\Xi(n) \leq 5$\smallskip
\end{itemize}

Naturally, we overview the difficulties currently preventing us from extending the approach to (nonzero) even targets.
\end{abstract}

\section{Introduction}

In 1973, Neil Sloane published a very short paper \cite{sloane} introducing an intriguing problem: Pick a decimal integer $n$ and multiply all its digits by each other. Repeat the process until a single digit $\Delta(n)$ is obtained. $\Delta(n)$ is called the \textsl{multiplicative digital root of $n$} or \textsl{the target of $n$}. The number of steps $\Xi(n)$ needed to reach $\Delta(n)$, called \textsl{the multiplicative persistence of $n$} or \textsl{the height of $n$} is conjectured to always be at most $11$.\smallskip

For instance, the target of $39$ is $4$, because:

$$39\rightarrow 3\times9=27 \rightarrow 2\times 7= 14 \rightarrow 1\times 4=4=\Delta(39)$$ 

Like many other very simple to state number-theoretic conjectures, the multiplicative persistence mystery resisted numerous explanation attempts \cite{ref1}, \cite{ref2}, \cite{ref4}, \cite{ref5}, \cite{ref6}, \cite{ref7}.\smallskip

In particular, the conjecture is known to hold at least up to $n=10^{20000}$ \cite{ref3}.

Addressing the Multiplicative persistence conjectures consists in studying the function $f : n \rightarrow f(n)$, where $f(n)$ is obtained by multiplying the digits of the the number $n$.\smallskip

$\Xi(n)$ is hence the smallest $k$ such that $f^k(n)\leq 9$.\smallskip

Note that $\Xi(n)$ is defined for all $n\in\mathbb{N}$: letting $n=\sum_{i=0}^r 10^i a_i$ (where the $0\leq a_i\leq 9$ are digits), $f(n)=\prod_{i=0}^r a_i<a_r\times 10^r\leq n$.\smallskip

$\mathcal F_n=\{f(n),f^2(n),f^3(n),\ldots\}$ is thus a positive decreasing sequence, and as such, it converges.\smallskip

Since $\mathcal F_n$ takes values in $\mathbb{N}$, $\mathcal F_n$ can only converge by reaching a fix-point $\Delta(n)$ and staying at it. However, the $\mathcal F_n$ is strictly decreasing while its values have at least two decimal digits.\smallskip

Finally, $\mathcal F_n$ converges converges toward a one-digit number\footnote{known as ``multiplicative digital root'' or ``target''.}, $\Delta(n)\leq 9$. Hence the notion of \textsl{multiplicative persistence}\footnote{or ``height''.}
 $\Xi(n)$ defined as the number of steps required to reach $\Delta(n)$.\smallskip

The following is a famous conjecture \cite{rg}:

\begin{conjecture}
\label{conjecture-main}
$\forall n\in \mathbb{N}$, $\Xi(n)\leq 11$.\smallskip
\end{conjecture}

\begin{figure}
    \centering
    \includegraphics{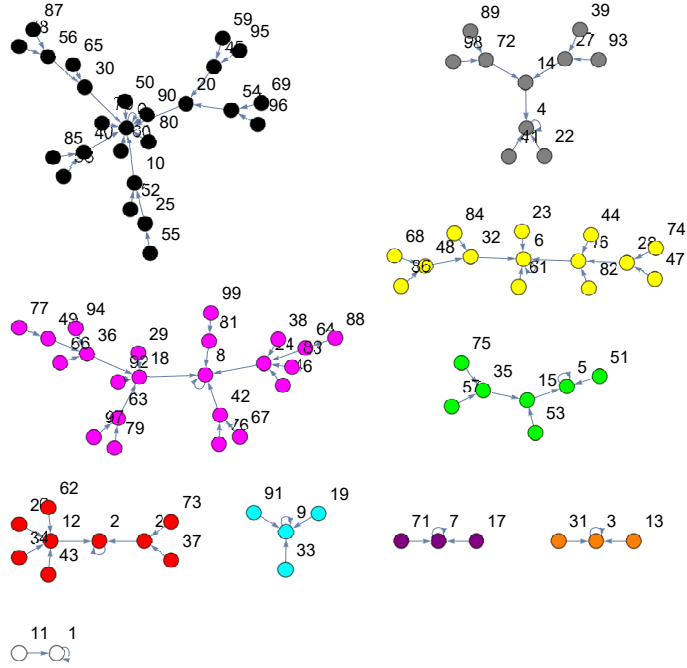}
    \caption{Genealogy for $0\leq n \leq 99$.}
    \label{fig:pers}
\end{figure}

In this work we prove the conjecture for all odd targets\footnote{i.e. $\Delta(n) \in \{1;\ 3;\ 5;\ 7;\ 9\}$} and provide bounds for $\Xi(n)$ depending on the value of $\Delta(n)$.

\subsection{Notations}
To present formulae concisely, we introduce the following compact notation:

$$\dtcs{a}{b}{c}{d}=2^a \times 3^b \times 5^c \times 7^d$$

When an exponent is zero we might just omit the corresponding entry in the notation, or replace it by a $0$ e.g.:

$$\dtcs{0}{b}{c}{d}=\tcs{b}{c}{d}= 3^b \times 5^c \times 7^d$$

$\ord_{n}(a)$ will denote the order of $a \bmod n$, i.e. the smallest positive integer $k$ such that $a^k\ \equiv\ 1 \bmod n$.\smallskip

In this paper, the term ``digit'' will exclusively refer to decimal digits.

Let $d$ be a digit, the shorthand notation $d_x$ will stand for a sequence of $x$ consecutive digits $d$, e.g.:

$$9_80_71=9999999900000001$$

To simplify notations we will denote by $\vec{x}$ a sequence of $(k+1)$ indexed variables starting with $x_0$ and ending with $x_k$, that is: $\vec{x}=\mbox{``}x_0,x_1,\ldots,x_k\mbox{''}$. Operations on vectors are to be understood component wise, e.g.:\smallskip

$$\vec{x}+7\vec{y}=\mbox{``}x_0+7y_0,x_1+7y_1,\ldots,x_k+7y_k\mbox{''}$$

Finally, we will also need the following definition:

\begin{definition}
$\dec{a}{b}{c}{d}$ denotes the set of decimal integers where the digits $3,5,7,9$ respectively appear $a,b,c,d$ times (at any position) with any number of $1$s.\smallskip
\end{definition}

The acronyms \dnv, \aad~and \anad~will respectively stand for ``\underline{d}o(es) \underline{n}ot \underline{v}erify'', ``\underline{a}re \underline{a}ll \underline{d}ifferent (from each other)'' and ``\underline{a}re \underline{n}ot \underline{a}ll \underline{d}ifferent''. e.g. ``$x=2$ \dnv~$x^2=5$'',
``\{1,2,3\} \aad'', ``$a=1$, $b=2$ and $c=1$ \anad.''

\section{Convergence Genealogies}

\textbf{(P1)} : Take $y$ an odd number. If $\exists x$ such that $f(x)=y$ then, as the product of the digits of $x$, $y$ can be written as $y = \dtcs{\alpha}{\beta}{\gamma}{\delta}$. Since $y$ is odd, $y =\tcs{\beta}{\gamma}{\delta}$. As an antecedent of $y$, $x$ belongs to one of the sets $\dec{\beta_1}{\gamma}{\delta}{\beta_2}$, with $\beta_1 + 2 \times \beta_2 = \beta$. Conversely, if $y =\tcs{\beta}{\gamma}{\delta}$ and $x \in \dec{\beta_1}{\gamma}{\delta}{\beta_2}$, with $\beta_1 + 2 \times \beta_2 = \beta$, then $f(x)=y$.\smallskip

Let us notice that if $\Delta(n)$ is odd, then $n$ has to be odd too: indeed, if $n$ would have been even, its last digit would be even as well, and thus $f(n)$ would be even, and so on.\smallskip

For a digit $d$, let us denote as tree of antecedents of $d$ the graph $A_d=(V_d,E_d)$ defined as follows:
\begin{itemize}
    \item $d \in V_d$
    \item If $s \in V_d$ and $x$ such that $f^2(x)=s$, then $f(x)\in V_d$ and $(s,f(x))\in E_d$
\end{itemize}

\textbf{(P2)} : Then, $\{n : \Delta(n)=d\} = \{n : f(n)\in V_d\}$, and $\Xi(n)$ for $n$ such that $\Delta(n)=d$ is the number of different nodes in the path connecting $d$ to $f(n)$.\smallskip

\textbf{(P1)} and \textbf{(P2)} together with the knowledge of $A_1$, $A_3$, $A_5$, $A_7$ and $A_9$ describe all numbers $n$ such that $\Delta(n)=d$ for a given odd digit $d$. We now want to prove that $A_1$, $A_3$, $A_5$, $A_7$ and $A_9$ are respectively the following graphs:

\begin{figure}
$B_1=(U_1,F_1)$ = \qquad
\begin{tikzpicture}
  [
    grow                    = right,
    sibling distance        = 4em,
    level distance          = 8em,
    edge from parent/.style = {draw, -latex},
    every node/.style       = {font=\footnotesize},
    sloped
  ]
  \node [node] {1}
  edge [loop above] ()
    ;
\end{tikzpicture}

$B_3=(U_3,F_3)$ = \qquad
\begin{tikzpicture}
  [
    grow                    = right,
    sibling distance        = 4em,
    level distance          = 8em,
    edge from parent/.style = {draw, -latex},
    every node/.style       = {font=\footnotesize},
    sloped
  ]
  \node [node] {3}
    edge [loop above] ()
    ;
\end{tikzpicture}
\\

$B_5=(U_5,F_5)$ =\\
\begin{tikzpicture}
\tikzstyle{level 1}=[sibling distance=8.5em,level distance = 6em, edge from parent/.style = {draw, -latex},every node/.style       = {font=\footnotesize},rotate=90]
  \tikzstyle{level 2}=[sibling distance=5.6em,level distance = 6em, edge from parent/.style = {draw, -latex},every node/.style       = {font=\footnotesize},minimum width=3em]
    \tikzstyle{level 3}=[sibling distance=2.8em,level distance = 7em, edge from parent/.style = {draw, -latex},every node/.style       = {font=\footnotesize},minimum width=4em]
  [
    grow                    = right,
    sibling distance        = 4em,
    level distance          = 6em,
    edge from parent/.style = {draw, -latex},
    every node/.style       = {font=\footnotesize},
    sloped
  ]
  \node [node] {5}
    child { node [node] {15}
      child { node [node] {315} 
          child { node [node] {3375} 
              child { node [node] {59535} }
          }
      }
      child { node [node] {135} }
      child { node [node] {35} 
          child { node [node] {1715} 
              child { node [node] {77175} }
          }
          child { node [node] {175} 
              child { node [node] {1575} }
          }
          child { node [node] {75} }
      }
    }
    edge [loop above] ()
    ;
\end{tikzpicture}

$B_7=(U_7,F_7)$ = \qquad
\begin{tikzpicture}
  [
    grow                    = right,
    sibling distance        = 4em,
    level distance          = 8em,
    edge from parent/.style = {draw, -latex},
    every node/.style       = {font=\footnotesize},
    sloped
  ]
  \node [node] {7}
    edge [loop above] ()
    ;
\end{tikzpicture}

$B_9=(U_9,F_9)$ = \qquad
\begin{tikzpicture}
  [
    grow                    = right,
    sibling distance        = 4em,
    level distance          = 8em,
    edge from parent/.style = {draw, -latex},
    every node/.style       = {font=\footnotesize},
    sloped
  ]
  \node [node] {9}
    edge [loop above] ()
    ;
\end{tikzpicture}
\end{figure}

\section{Establishing the Equations}
The five graphs above are respectively sub-graphs of $A_1$, $A_3$, $A_5$, $A_7$ and $A_9$. To prove that they are also equal to said graphs, we need to prove for $d\in \{1;\ 3;\ 5;\ 7;\ 9\}$ that for each $s\in U_d$, if $x$ is such that $f^2(x)=s$, then $f(x)\in U_d$.\smallskip

For example take $d=1$ and $s=1$. Let us consider $x$ such that $f^2(x)=s$. Because $1=\tcs{0}{0}{0}$, \textbf{(P1)} gives that 
$f(x) \in \dec{0}{0}{0}{0}$, meaning that $f(x)$ is only composed of $1$s. Therefore, there exists $a$ (corresponding to the number of digits of $f(x)$) such that $f(x)=\frac{10^a-1}{9}$. On the other hand, $f(x)=\tcs{\beta}{\gamma}{\delta}$, and $f(x)$ has neither $0$ nor $5$ as digits since $f^2(x)=1$, so we have $\gamma = 0$. We thus want to solve $\frac{10^a-1}{9}=\dtcs{0}{\beta}{0}{\delta}$, or equivalently: $10^a-1=\dtcs{0}{u}{0}{w}$.\\

For another example, take $d=5$ and $s=315$. Let us consider $x$ such that $f^2(x)=s$. Because $315=\tcs{2}{1}{1}$, \textbf{(P1)} gives that either $f(x) \in \dec{2}{1}{1}{0}$ or $f(x) \in \dec{0}{1}{1}{1}$. Therefore, either
\begin{itemize}
    \item there exist $a$ (corresponding to the number of digits of $f(x)$), and $b,c,d,e$ (corresponding respectively to the positions\footnote{Position 0 corresponds to the least significant digit.} of the digits $3$, $3$, $5$ and $7$ in $f(x)$) such that $f(x) = \frac{10^a-1}{9} + (3-1) \times 10^b + (3-1) \times 10^c + (5-1) \times 10^d + (7-1) \times 10^e$ (with $a>b,c,d,e\geq 0$ and $b,c,d,e$ all different because of what $a,b,c,d,e$ represent),
    \item or there exist $a$ and $b,c,d$ such that $f(x) = \frac{10^a-1}{9} + (5-1) \times 10^b + (7-1) \times 10^c + (9-1) \times 10^d$.
\end{itemize}
On the other hand, $f(x)=\tcs{\beta}{\gamma}{\delta}$. We thus want to solve
\begin{itemize}
    \item $10^a + 18 \times 10^b + 18 \times 10^c + 36 \times 10^d + 54 \times 10^e - 1 = \tcs{u}{v}{w}$
    \item and $10^a + 36 \times 10^b + 54 \times 10^c + 72 \times 10^d - 1 = \tcs{u}{v}{w}$
\end{itemize}

\subsection{About the Possible Values of $v=\gamma$, the Power of $5$}

If $d\neq 5$ is an odd digit, if $s \in U_d$ and $x$ such that $f^2(x)=s$, $f(x)=\tcs{\beta}{\gamma}{\delta}$ with $\gamma = 0$ since $f(x)$ would otherwise end with $0$ or $5$, which would make $f^k(x)=d$ for some $k$ impossible.\smallskip

For $d=5$ however, if $s \in U_5$ and $x$ such that $f^2(x)=s$, we initially only know that $f(x)=\tcs{\beta}{\gamma} {\delta}$.\smallskip

Let us look at the values $\tcs{\beta}{\gamma}{\delta}$ takes if $\gamma \geq 5$: odd multiples of $5^5$ have their values modulo $10^5$ in\footnote{This set is formed of $5^5\times (2i+1)$ for $0\leq i \leq 15$.}:

\begin{center}
\begin{tabular}{rrrrrrrrrr}
  &\{3125,&9375,&15625,&21875,&28125,&34375,&40625,&46875,\\
    &53125,&59375,&65625,&71875,&78125,&84375,&90625,&96875\}\\
\end{tabular}  
\end{center}

Therefore, if $f(x)=\tcs{\beta}{\gamma}{\delta}$ with $\gamma \geq 5$, either $f(x)=9375$ in $\mathbb{Z}$ or $f(x) \bmod 10^5 = 59375=\dtcs{0}{5}{1}{2}$. However, $f(9375)=945 \notin U_5$. Furthermore, if $f(x) \bmod 10^5 = 59375$, $f^2(x)$ is a multiple of $4725$, but $4725$ divides no element of $U_5$. From that, we know that $\gamma \leq 4$.\smallskip

Similarly, odd multiples of $5^4$ have their values modulo $10^4$ in \footnote{This set is formed of $5^4\times (2i+1)$ for $0\leq i \leq 7$.}:

$$\{625,1875,3125,4375,5625,6875,8125,9375\}$$

Thus, if $\gamma = 4$, $s=f^2(x)$ is a multiple of $f(9375)=945$. Since $59535$ is the only element of $U_5$ that $945$ divides, $s=59535$. Note that knowing $f(x) \bmod 10^4 = 9375$ gives us the positions of some digits $9,3,7,5$ in $f(x)$.\smallskip

Likewise, if $\gamma = 3$, $f(x) \bmod 10^3 = 375$ and $f(375)=105$ divides $s=f^2(x)$, so $s \in \{315, 1575,59535, 77175\}=\dtcs{0}{2}{1}{1}\times\{1, 5, 189, 245\}$.\smallskip

If $\gamma = 2$, $f(75)=35$ divides $s$, so $s \in \{35,175,315,1575,1715,59535,77175\}$.\smallskip

Furthermore, because of the values $3^{\beta} \times 7^{\delta}$ takes modulo $100$ (the tens digit is always even), $\gamma \geq 1$.

\subsection{Listing all Equations}
To summarise, let us take $d$ an odd digit, $s \in U_d$ and $x$ such that $f^2(x)=s$. Then $f(x)=\tcs{\beta}{\gamma}{\delta}$ with $\gamma = 0$ if $d \neq 5$, and $\gamma$ taking at most four different values from $\{1,2,3,4\}$ depending on $s$ if $d=5$. Also $f(x)\in \dec{h_1}{i}{j}{h_2}$, with $h_1 + 2 \times h_2 = h$ if $s=\tcs{h}{i}{j}$.\smallskip

For all $d, s$ we can range over all possible split $(h_1,h_2)$ and all possible $\gamma$ to establish a finite list of equations that $f(x)$ must satisfy. The number of equations for all possible $d$ are as follows:

\begin{center}
\begin{tabular}{|c||c|c|c|c|c|}\hline
~~value of $d$~~&~~1~~&~~3~~&~~5~~&~~7~~&~~9~~\\\hline
~~number of equations~~&~~1~~&~~1~~&~39~&~~1~~&~~2~~\\\hline
\end{tabular}
\end{center}

Notice that when $\gamma \neq 0$, the knowledge of $f(x) \bmod 10^{\gamma}$ gives the position of the $\gamma$ last digits of $f(x)$, and they are always such that we are able to divide both sides of the equation by $5^{\gamma}$.\smallskip

As a consequence, all equations can be put in the form $\mbox{\textsf{LC}}(\vec{a})=\dtcs{0}{u}{0}{w}$ where $\mbox{\textsf{LC}}(\vec{a})$ is a linear combination of $10^{a_0}$, $10^{a_1}$, \ldots, $10^{a_k}$. Precisely, they are defined by the formula:
$$2^h\times (10^{a_0}+18\sum_{i=1}^k c_i\times10^{a_i})+\tau_h=
\dtcs{0}{u}{0}{w}$$
where
{\small
\begin{center}
\begin{tabular}{|c||c|c|c|c|c|}\hline
~~~$h=\gamma$~~~&~~~$0$~~~& ~~~$1$~~~&~~~$2$~~~&~~~$3$~~~&~~~$4$~~~\\\hline
~~~$\tau_h$~~~&~~~$-1$~~~& ~~~$7$~~~&~~~$23$~~~&~~~$19$~~~&~~~$119$~~~\\\hline
\end{tabular}
\end{center}}
\noindent and where the other parameters are given in Appendix \ref{App1}.\smallskip

The solutions that interest us must verify \textbf{(R)} : $a_0>a_1,\ldots,a_k$ and $a_1,\ldots,a_k$ \aad~and non-negative. The following section describes a resolution algorithm for solving this kind of equations. It outputs a finite list of solutions from which one can verify that the only ones satisfying \textbf{(R)} correspond to values $f(x)$ belonging to $U_d$.

\section{Resolution Algorithm}

\subsection{The General Principle}
\label{sec:general-principle}

We want to solve \textbf{(E)} : $\mbox{\textsf{LC}}(\vec{a})=\dtcs{0}{u}{0}{w}$ where $\mbox{\textsf{LC}}(\vec{a})$ is a linear combination of $10^{a_0}$, $10^{a_1}$, \ldots, $10^{a_k}$, $a_0>a_1,\ldots,a_k$ and $a_1,\ldots,a_k$ \aad.\smallskip

For some $t>0$ and $y,z \leq t$, assume that we know for $(\vec{a},u,w)$ a solution of \textbf{(E)} the residues $(\vec{a'}) = (\vec{a}) \bmod t$, $u' = u \bmod \ord_{\dtcs{y}{0}{z}{0}}(3)$, and $w' = w \bmod \ord_{\dtcs{y}{0}{z}{0}}(7)$.

Because $\mbox{\textsf{LC}}(\vec{a})$ is a linear combination of $10^{a_0},10^{a_1},\ldots,10^{a_k}$, the knowledge of $(\vec{a}) \bmod t$ gives at most $2^{k+1}$ possible values for $\mbox{\textsf{LC}}(\vec{a}) \bmod \dtcs{y}{0}{z}{0}$. Indeed, since $a_i=a'_i+q_i \times t$ with $0 \leq a'_i < t$, we have
$$10^{a_i} = 10^{a'_i} \times (10^t)^{q_i} \equiv 10^{a'_i} \times (0)^{q_i} \bmod{\dtcs{y}{0}{z}{0}}.$$
Thus, for each $i$, $10^{a_i} \bmod \dtcs{y}{0}{z}{0}$ can take only two values: either $10^{a'_i}$ (in which case $a_i$ is known in $\mathbb{Z}$) or $0$.\smallskip

Note that the knowledge of $u' = u \bmod \ord_{\dtcs{y}{0}{z}{0}}(3)$ and $w' = w \bmod \ord_{\dtcs{y}{0}{z}{0}}(7)$ gives us $\dtcs{0}{u}{0}{w} \bmod \dtcs{y}{0}{z}{0}$.\smallskip

Considering the equation \textbf{(E)} modulo $\dtcs{y}{0}{z}{0}$, we are able, for each $(\vec{b}) = (b_0,\ldots,b_k) \in \{0,1\}^{k+1}$, to compute $\textsf{LC}(\vec{a}) \bmod \dtcs{y}{0}{z}{0} = \sum_i b_i 10^{a'_i} \bmod \dtcs{y}{0}{z}{0}$ and check whether it is equal to $3^{u'} \times 7^{w'} \bmod \dtcs{y}{0}{z}{0}$.

We want $y$ and $z$ to be high enough such that the congruence modulo $\dtcs{y}{0}{z}{0}$ is satisfied for only a few $(\vec{b})$. This set of complying solutions can then be further reduced. Indeed, for each surviving $(\vec{b})$ one can check whether the partial set of known $a_i$ ($a_i$ is known in $\mathbb{Z}$ for all $i$ such that $b_i=1$) also complies with requirement \textbf{(R)}. The hope is that each tuple $(\vec{a'},u',w')$ we may start from, either yields a uniquely determined ($b_i=1$ for all $i$) solution $f(x)$ belonging to $U_d$, or no solution at all.




\subsection{Application}

Let us e.g. chose $(t,y,z)=(12,9,6)$. $\ord_{\dtcs{9}{0}{6}{0}}(3) = 
\dtcs{7}{0}{5}{0}$ and $\ord_{
\dtcs{9}{0}{6}{0}}(7) = \dtcs{6}{0}{4}{0}$.\smallskip

Let us notice that $10^{12} - 1 = \dtcs{0}{3}{0}{1} \times 11 \times 13 \times 37 \times 101 \times 9901$. Denote $m_{12} := \frac{10^{12} - 1}{3^3 \times 7} = 11 \times 13 \times 37 \times 101 \times 9901$. $m_{12}$ has the two following interesting properties:
\begin{itemize}
    \item $\ord_{m_{12}}(10)=12$
    \item and $(\dtcs{0}{u'}{0}{w'} \equiv 1 \bmod{m_{12}}) \iff
    \left \{\begin{array}{l}
u' \equiv 0 \bmod{9900}\\
w' \equiv 0 \bmod{900}
\end{array}
\right.$.

\end{itemize}

Assume that $(\vec{a},u,w)$ is a solution of \textbf{(E)}. Let us look at all $0 \leq a'_1, a'_2, \ldots a'_k < 12$ such that $\mbox{\textsf{LC}}(\vec{a'}) \equiv 3^{u'} \times 7^{w'} \bmod{m_{12}}$ for some $0 \leq u' < 9900$ and $0 \leq w' < 900$. Thanks to the two properties of $m_{12}$ mentioned here above, we can thus reduce the possible values of $(\vec{a}) \bmod 12$ and get matching values for $u \bmod 9900$ and $w \bmod 900$.\smallskip

At this point we have a set of possible values for $(\vec{a}) \bmod 12$, and for each of them we know the corresponding values of  $u' = u \bmod (11 \times \dtcs{2}{2}{2}{\phantom{0}})$ and $w' = w \bmod \dtcs{2}{2}{2}{\phantom{0}}$. Remember that our goal is to know $u$ and $w$ modulo $\dtcs{7}{0}{5}{0}$ and $\dtcs{6}{0}{4}{0}$ respectively. To improve our knowledge of $u$ and $w$ our strategy is to lift from equation \textbf{(E)} considered modulo $m_{12}$ to \textbf{(E)} considered modulo $m_{24} := \frac{10^{24} - 1}{3^3 \times 7} = m_{12} \times 73 \times 137 \times 9_40_31$.
\smallskip

For each value $(\vec{a})$ can take modulo 12, we obtain $2^{k+1}$ different values for $(\vec{a}) \bmod 24$. Consider $p = 73$ a prime divisor of $m_{24}$ and observe that $\ord_{73}(3)=\dtcs{2}{1}{0}{0}$ and $\ord_{73}(7)=\dtcs{3}{1}{0}{0}$. Since $\ord_{73}(7)=\dtcs{3}{1}{0}{0}$ we can improve our knowledge of $w$ from $w \bmod \dtcs{2}{2}{2}{\phantom{0}}$ to $w \bmod \dtcs{3}{2}{2}{\phantom{0}}$. This is due to the principle described in Section~\ref{sec:improving-knowledge-u-w}. Considering other prime divisors of $m_{24}$ progressively further improves our knowledge of $u$ and $w$. As this knowledge is not yet sufficient, we eventually have to lift upper from $m_{24}$ to $m_{48} := \frac{10^{48} - 1}{3^3 \times 7}$ and exploit other prime divisors of $m_{48}$ up to obtaining our sufficient knowledge of $u \bmod \dtcs{7}{0}{5}{0}$ and $w \bmod \dtcs{6}{0}{4}{0}$.
\smallskip





\subsection{Improving our knowledge of $u$ and $w$}
\label{sec:improving-knowledge-u-w}

Suppose that we know the value of $\mbox{\textsf{LC}}(\vec{a}) \bmod p$. Suppose that we have $m_u$ and $m_w$ such that we know of $0 \leq u' < m_u$ and $0 \leq w' < m_w$ verifying $u = u' + h_u m_u$ and $w = w' + h_w m_w$. Suppose that $q$, $n_u=\lambda \times m_u$ and $n_w=\mu \times m_w$ are such that $\ord_p(3)$ divides $q \times n_u$ and $\ord_p(7)$ divides $q \times n_w$.\smallskip

We are searching for $0 \leq u'' < n_u$ and $0 \leq w'' < n_w$ such that $u = u'' + k_u n_u$ and $w = w'' + k_w n_w$ for some $k_u$ and $k_w$. Let us write $h_u = h'_u + c \times \lambda$.\smallskip
$$u = u' + h_u m_u = u' + h'_u m_u + c \lambda m_u = u' + h'_u m_u + c n_u$$

$u''$ is therefore of the form $u' + h'_u m_u$ for some $0 \leq h'_u < \lambda$. Similarly, $w''$ is of the form $w' + h'_w m_w$ for some $0 \leq h'_w < \mu$.\smallskip

Looking at $(\mbox{\textbf{E}})^q$ $\bmod p$ gives us\smallskip

$\begin{aligned}
    (\mbox{\textsf{LC}}(\vec{a}))^q &= (3^{u} \times 7^{w})^q=\dtcs{0}{u}{0}{w}^q=\dtcs{0}{u'' + k_u n_u}{0}{w'' + k_w n_w}^q= \dtcs{0}{u' + h'_u m_u + k_u n_u}{0}{w' + h'_w m_w + k_w n_w}^q\\
        & = \dtcs{0}{u' + h'_u m_u}{0}{w' + h'_w m_w}^q \times \dtcs{0}{k_u q n_u}{0}{ k_w q n_w}
        \equiv (3^{u' + h'_u m_u} \times 7^{w' + h'_w m_w})^q \bmod{p}
\end{aligned}$


since $\ord_p(3)$ divides $q \times n_u$ and $\ord_p(7)$ divides $q \times n_v$.\smallskip

We thus have a way to search for $0 \leq h'_u < \lambda$ and $0 \leq h'_w < \mu$ to transform our knowledge of $u \bmod m_u$ and $w \bmod m_w$ into the knowledge of $u \bmod n_u$ and $w \bmod n_w$.

\subsection{The Equation resolution algorithm}

Let us start with Algorithm~\ref{aMod12} to find $(\vec{a}) \bmod 12$ and the matching $u \bmod 9900$ and $w \bmod 900$.

\begin{algorithm}
\caption{Finding $(\vec{a'}) = (\vec{a}) \bmod 12$, $u' = u \bmod 9900$ and $w' = w \bmod 900$}\label{aMod12}
\begin{algorithmic}[1]
\Procedure{find}{$g,t$}\Comment{Dichotomy to find whether $g$ has matching $u'$ and $w'$}
    \State{$d \gets 0$}
    \State{$f \gets$length($t$)-1}
    \While{$d \leq f$}
        \State{$m \gets (d+f) / 2$}
        \State{$(y,u',w') \gets t[m]$}
        \If{$y=g$}
            \State{\textbf{return} (\textsf{True},$u',w'$)}
        \ElsIf{$g>y$}
            \State{$d \gets m+1$}
        \Else
            \State{$f \gets m-1$}
        \EndIf
    \EndWhile
    \State{ \textbf{return} (\textsf{False},$0,0$)}
\EndProcedure
\Procedure{aMod12}{\textsf{LC}}
    \State{$m_{12} \gets \frac{10^{12} - 1}{3^3 \times 7}$}
    \State{leftModM12 $\gets [\ ]$} 
    \ForAll{$(u',w') \in \llbracket 0;9900 - 1 \rrbracket \times \llbracket 0;900 - 1 \rrbracket$}
        \State{$g \gets 3^{u'} \times 7^{w'} \bmod m_{12}$}
        \State{leftModM12.append(($g,u',w'$))}
    \EndFor
    \State{leftModM12.sort()}
    \State{$S_{12} \gets \emptyset$}
    \ForAll{$(\vec{a'}) \in \llbracket 0;11 \rrbracket^{k+1}$}\Comment{non-optimal}
        \State{$g \gets \mathsf{LC}(\vec{a'}) \bmod m_{12}$}
        \State{found,$u'$,$w'$ $\gets$  \textsc{find}($g$,leftModM12)}
        \If{found}
            \State{$S_{12} \gets S_{12} \cup \{(\vec{a'},u',w')\}$}
        \EndIf
    \EndFor
    \State \textbf{return} $S_{12}$
\EndProcedure
\end{algorithmic}
\end{algorithm}

Notice that trying all $(\vec{a'}) \in \llbracket 0;11 \rrbracket^{k+1}$ is not always necessary: indeed, some tuple may be equivalent except for ordering, and it is possible to treat only one of those by adding conditions such as $a'_1 < a'_2$ for example.\\
\\
For Algorithm~\ref{upgradeuw}, let us make the following assumptions:
\begin{itemize}
    \item we know $(\vec{a}) \bmod t$ where $t$ is a multiple of $\ord_p(10)$ and we hence know $g = \mbox{\textsf{LC}}(\vec{a}) \bmod p$
    \item we know $u \bmod m_u$ and $w \bmod m_w$
    \item $q$, $n_u=\lambda \times m_u$ and $n_w=\mu \times m_w$ are such that $\ord_p(3)$ divides $q \times n_u$ and $\ord_p(7)$ divides $q \times n_w$
\end{itemize}

\begin{algorithm}
\caption{Improving our knowledge of $u$ and $w$}\label{upgradeuw}
\begin{algorithmic}[2]
\Procedure{updgradeuw}{$p,g,u',w',m_u,m_w,\lambda,\mu,q$}
    \State{$R' \gets \emptyset$}
    \ForAll{$h'_u \in \llbracket 0 ; \lambda - 1 \rrbracket$}
        \ForAll{$h'_w \in \llbracket 0 ; \mu - 1 \rrbracket$}
            \If{$g^q \bmod p = (3^{u' + h'_u m_u} \times 7^{w' + h'_w m_w})^q \bmod p$}
                \State{$R' \gets R' \cup \{u' + h'_u m_u,w' + h'_w m_w\}$}
            \EndIf
        \EndFor
    \EndFor
    \State \textbf{return} $R'$
\EndProcedure
\end{algorithmic}
\end{algorithm}

Algorithm~\ref{upgradeuw} used with $g = \mbox{\textsf{LC}}(\vec{a}) \bmod p$ for one candidate value of $(\vec{a})$ modulo $t$ returns the matching possible values for $u \bmod \lambda m_u$ and $w \bmod \mu m_w$.\smallskip



To gain knowledge of $u \bmod \dtcs{7}{0}{5}{0}$ and $w \bmod \dtcs{6}{0}{4}{0}$, we follow the path described in Algorithm~\ref{learn} where $\mbox{\textsf{Learn}}(i)$ denotes the operation of using the parameters of line $i$ of Table \ref{para} to learn the possible matches of $u,w$ modulo the last two columns of Table \ref{para}.
\smallskip

\begin{algorithm}
\caption{Learning Path}\label{learn}
\begin{algorithmic}[2]
    \State{Use \textsc{aMod12} to get:}
    \State{~~~~~~~~~~~~~~~~~~~~~~~~~~~~~~$\vec{a}_{12}=(\vec{a}) \bmod 12$}
    \State{~~~~~~~~~~~~~~~~~~~~~~~~~~~~~~and the matching $u \bmod 9900$ and $w\bmod 900$}
    \State{Transform the $\vec{a}_{12}$ and matching $u$ and $w$ into a $2^{k+1}$ bigger set of possible:}
    \State{~~~~~~~~~~~~~~~~~~~~~~~~~~~~~~$\vec{a}_{24}=(\vec{a}) \bmod 24$}
    \State{~~~~~~~~~~~~~~~~~~~~~~~~~~~~~~and the matching $u \bmod 9900$ and $w\bmod 900$}
    \ForAll{$\vec{a}_{24}$ and matching $u \bmod 9900$ and $w\bmod 900$ candidates:}
        \State{$\mbox{\textsf{Learn}}(1)$}
        \State{$\mbox{\textsf{Learn}}(2)$}
        \State{$\mbox{\textsf{Learn}}(3)$} 
        \ForAll{$\vec{a}_{48}$ resulting from lifting from modulo $m_{24}$ to modulo $m_{48}$}
        \State{$\mbox{\textsf{Learn}}(4)$}
        \State{$\mbox{\textsf{Learn}}(5)$}
        \EndFor
    \EndFor
\end{algorithmic}
\end{algorithm}

\begin{table}[]
\begin{center}
\begin{tabular}{|c|c|c|c|c|c|c|}\hline
~~~$i$~~~&\den{$p_i$}&\den{$q_i$}&\den{$\lambda_i$}&\den{$\mu_i$}&\den{learn $u~\bmod$}&\den{learn $w~\bmod$}\\[0.1cm]\hline\hline
\den{$1$}&\den{$73$}&\den{$1$}&\den{$1$}&\den{$2$}&\den{$\dtcs{2}{2}{2}{\phantom{0}}\times 11$} &\den{$\dtcs{3}{2}{2}{\phantom{0}}$}\\[0.1cm]\hline
\den{$2$}&\den{$137$}&\den{$17$}&\den{$2$}&\den{$1$}&\den{$
    \dtcs{3}{2}{2}{\phantom{0}} \times 11$}&\den{$\dtcs{3}{2}{2}{\phantom{0}}$}\\[0.1cm]\hline 
\den{$3$}&\den{$9_40_31$}&\den{$11\times 101$}&\den{$5^2$}&\den{$2$}&  \den{$\dtcs{3}{2}{4}{\phantom{0}}\times 11$}&\den{$\dtcs{4}{2}{2}{\phantom{0}}$}\\[0.1cm]\hline
\den{$4$}&\den{$17$}&\den{$1$}&\den{$2$}&\den{$1$}&\den{$\dtcs{4}{2}{4}{\phantom{0}}\times 11$} &\den{$\dtcs{4}{2}{2}{\phantom{0}}$}\\[0.1cm]\hline  
\den{$5$}&\den{$9_80_71$}&\den{$\frac{p_5-1}{14_20_4}$}&\den{$2^3 \times 5$}&\den{$2^4 \times  5^2$} &\den{$\dtcs{7}{2}{5}{\phantom{0}}\times 11$} &\den{$\dtcs{8}{2}{4}{\phantom{0}}$}\\[0.1cm]\hline
\end{tabular}
\end{center}
\caption{Parameters for function $\mbox{\textsf{Learn}}(i)$}
\label{para}
\end{table}

This uses the following facts:
\vspace{-0.5cm}
\subsubsection{~~~~Fact 1:} cf. Table \ref{tab:template}.
\vspace{-0.2cm}
\begin{table}[!ht]
\begin{center}
\begin{tabular}{|c|c|c|c|}\hline
~~~$i$~~~&\den{$p_i$} &\den{$\ord_{p_i}(3)$}&\den{$\ord_{p_i}(7)$}\\[0.1cm]\hline\hline
~~~$1$~~~&\den{$73$}&\den{$2^2 \times 3$}&\den{$2^3 \times 3$}\\[0.1cm]\hline
~~~$2$~~~&\den{$137$}&\den{$2^3 \times 17$}&\den{$2^2 \times 17$}\\[0.1cm]\hline
~~~$3$~~~&\den{$9_40_31$}&\den{$\dtcs{3}{1}{4}{\phantom{0}}\times 11 \times 101$}& \den{$\dtcs{4}{2}{2}{\phantom{0}} \times 11$}\\[0.1cm]\hline
~~~$4$~~~&\den{$17$}&\den{$2^4$}&\den{$2^4$}\\[0.1cm]\hline
~~~$5$~~~&\den{$9_80_71$}&\den{$\dtcs{7}{1}{8}{\phantom{0}} \times 11 \times 73 \times 101 \times 137$}&\den{$2 \times 3 \times \ord_{p_5}(3)$}\\[0.1cm]\hline
\end{tabular}
\end{center}
\caption{Parameter values for the learning steps.}
\label{tab:template}
\end{table}

\subsubsection{~~~~Fact 2:}~~~$ m_{12} :=\frac{10^{12} - 1}{3^3 \times 7} = 11 \times 13 \times 37 \times 101 \times 9901$
\vspace{-0.2cm}
\subsubsection{~~~~Fact 3:}~~~$ m_{24} :=\frac{10^{24} - 1}{3^3 \times 7} = m_{12} \times p_1 \times p_2 \times p_3$\smallskip

Hence $\ord_p(10)$ divides $24$ for $p \in \{p_1, p_2, p_3\}$.
\vspace{-0.2cm}

\subsubsection{~~~~Fact 4:}~~~$ m_{48} := \frac{10^{48} - 1}{3^3 \times 7} = m_{24} \times p_4\times p_5 \times 5882353$\smallskip

Hence $\ord_p(10)$ divides $48$ for $p \in \{p_4,p_5\}$.


\bigskip
After the last \textsf{Learn}(5) phase we have solutions to \textbf{(E)} of the form $(\vec{a}_{48},u',w')$ where $\vec{a}_{48} = \vec{a} \bmod 48$, $u' = u \bmod (\dtcs{7}{2}{5}{\phantom{0}}\times 11)$ and $w' = w \bmod \dtcs{8}{2}{4}{\phantom{0}}$. Since $\ord_{\dtcs{9}{0}{6}{0}}(3) = \dtcs{7}{0}{5}{0}$ divides $\dtcs{7}{2}{5}{\phantom{0}}\times 11$ and $\ord_{\dtcs{9}{0}{6}{0}}(7) = \dtcs{6}{0}{4}{0}$ divides $\dtcs{8}{2}{4}{\phantom{0}}$, we can now consider equation \textbf{(E)} modulo $\dtcs{9}{0}{6}{0}$. Denote $\vec{a'} = \vec{a}_{48} \bmod 12$ with $\vec{a} = \vec{a'} + \vec{h'} \times 12$. As explained in Section~\ref{sec:general-principle}, when taken modulo $\dtcs{9}{0}{6}{0}$ each term $10^{a_i}$ of \textsf{LC}($\vec{a})$ may reduce either to $10^{a'_i}$ or to 0 depending whether $h'_i = 0$ or not. The first case corresponds to $a_i = a'_i < 12$ and $a_i$ is known in $\mathbb{Z}$, while the second case corresponds to $a_i \geq 12$ and $a_i$ is not uniquely determined.\smallskip

Given $\vec{a'}$, $u'$ and $w'$ Algorithm~\ref{toZ} exhausts all possible alternatives for each $a_i$ (whether $a_i < 12$ or not) and checks the congruence modulo $m = \dtcs{9}{0}{6}{0}$. For each solution that satisfies the congruence, if it returns $(a'_i,1)$ it means that $a_i$ is not known in $\mathbb{Z}$, but if it returns $(a'_i,0)$ it is known that $a_i = a'_i$. Then the set of known $a_i$ can be considered to check whether it violates the requirement \textbf{(R)}.


\begin{algorithm}
\caption{Getting $(\vec{a})$ in $\mathbb{Z}$}\label{toZ}
\begin{algorithmic}[2]
\Procedure{$\mbox{to}\mathbb{Z}$}{\textsf{LC},$\vec{a'},u',w'$}
    \State{$m \gets 2^{9} \times 5^6$}
    \State{$R \gets \emptyset$}
    \ForAll{$\vec{b} \in \{0,1\}^{k+1}$}
        \If{\textsf{LC}$(\vec{a'}+\vec{b} \times 12) \bmod m = 3^{u'} \times 7^{w'} \bmod m$}
            \State{$R \gets R \cup \{((a_0,b_0),(a_1,b_1),\ldots,(a_k,b_k),u',w')\}$}
        \EndIf
    \EndFor
    \State \textbf{return} $R$
\EndProcedure
\end{algorithmic}
\end{algorithm}

\section{Dealing with Nonzero Even Targets}

Having computationally demonstrated Conjecture~\ref{conjecture-main} for odd targets, a natural question that arises is whether it is possible to demonstrates the same for $d \in \{2;4;6;8\}$\footnote{The particular case $d=0$ can not be treated by our method.}. In this section we give general research directions about the computational difficulty of this task and provide shortcuts and observations that may be useful to reduce the complexity for anyone tempted to take over the challenge.\smallskip

Given $d$ and the graph $B_d=(U_d,F_d)$, we have to prove, for each $s \in U_d$, that if $f^2(x)=s$ then $f(x) \in U_d$. Given $s$, we have to solve as many different equations as the number of ways to express $s$ as a product of digits. As an example, for $d=2$ and $s=112$, expressing $s$ as the product $4 \times 4 \times 7$ leads to the following equation: $10^{a_0} + 27 \times 10^{a_1} + 27 \times 10^{a_2} + 54 \times 10^{a_3} = 2^t \times 3^u \times 7^w$.
The number of vertices in $U_d$ and the total number of equations they produce are thus parameters related to the difficulty of proving the conjecture for $d$. Nevertheless, all equations are not equally difficult to solve as we have to exhaust all $(k+1)$-tuples $(a_0,\ldots,a_k)$. The number $k$ of terms when expressing $s$ as a product of digits is thus particularly important. Consequently, the largest $k$ value, which relates to the computational work to solve the most difficult equation, is a more relevant parameter than the number of vertices or equations.\smallskip

Table~\ref{table-difficulty-parameters} gives for each even target the number of vertices in $U_d$, the total number of equations to solve, and the maximal number of terms $(k+1)$ for the left part of an equation. One can notice that for each even $d$, the number of equations to solve and the maximal $(k+1)$ to deal with are much more important than for odd targets (the maximal value of $(k+1)$ was merely equal to 8 for $d=5$ and $s=59535$). Given that each $a_i$ is defined modulo $12$, it seems totally out of reach to exhaust $12^{30}\simeq 2^{107}$ tuples in the most favorable case. Though, $d=2$ and $d=4$ are the most promising targets in terms of difficulty as there seems to be a gap with $d=6$ and $d=8$ both regarding the number of equations and its maximal difficulty.\smallskip

Should one want to prove the conjecture for $d=2$ or $d=4$, whose graphs $B_2$ and $B_4$ are given in Appendix~\ref{app-U2-U4}, we do have a number of technical observations (omitted here for lack of space) allowing to noticeably reduce the complexity of the exhaustive search needed in the first phase. Section~\ref{sec-power-of-two} considers the power of two that appears in the right part of the equations with a hope to reduce its negative impact on the filtering strength of the first phase.\smallskip

\begin{table}
\begin{center}
\begin{tabular}{|c|c|c|c|}\hline
~~Target~~&~~Number of~~&~~Total number~~&~~$(k+1)$ of the most~~\\
~~$d$~~&~~vertices~~&~~of equations~~&~~difficult equation~~\\[0.1cm]\hline\hline
~~2~~&~~33~~&~~1117~~&\den{30, for $s=\dtcs{26}{3}{0}{0}$}\\[0.1cm]\hline
~~4~~&~~9~~&~~1062~~&\den{32, for $s=\dtcs{23}{7}{0}{1}$}\\[0.1cm]\hline
~~6~~&~~84~~&~~6377~~&\den{37, for $s=\dtcs{24}{6}{0}{6}$}\\[0.1cm]\hline
~~8~~&~~51~~&~~4774~~&\den{45, for $s=\dtcs{39}{3}{0}{2}$}\\[0.1cm]\hline
\end{tabular}
\end{center}
\caption{Complexity parameters for $d \in \{2;4;6;8\}$.}
\label{table-difficulty-parameters}
\end{table}


\subsection{About the Possible Values of the Power of $2$}
\label{sec-power-of-two}

Given $d \in \{2;4;6;8\}$ and $s \in U_d$, we want to prove that if there exists $x$ such that $f(f(x))=s$ then $f(x) \in U_d$. Let $n_2,n_3\ldots,n_9$ such that $s=\prod_{i=2}^9 i^{n_i}$ and $\mathcal{A} := \decall{n_2}{n_3}{n_4}{n_5}{n_6}{n_7}{n_8}{n_9}$. Then $f(x)$ must simultaneously belong to $\mathcal{A}$ and be of the form $\dtcs{t}{u}{0}{w}$.

Compared to the case of odd $d$, the presence of the new power of $2$ -- which is a priori unbounded -- may result in a much weaker filter. Nevertheless we have noticed that for all $\mathcal{A}$ of interest to this paper, the power of two of all $x \in \mathcal{A}$ is actually bounded. This motivates the following conjecture.

\begin{conjecture}
Let $\mathcal{A} := \decall{n_2}{n_3}{n_4}{n_5}{n_6}{n_7}{n_8}{n_9}$ for some integers $n_2,\ldots,n_9$. Then there exists $a$ such that the maximum power of two in any $x \in \mathcal{A}$ is $2^a$.\smallskip
\end{conjecture}

We do not provide any proof of this conjecture, but Lemma~\ref{lemma-bound} allows to determine such bound $a$ for some specific sets $\mathcal{A}$.



\begin{lemma}
Let $\mathcal{A} := \decall{n_2}{n_3}{n_4}{n_5}{n_6}{n_7}{n_8}{n_9}$ for some integers $n_2,\ldots,n_9$. Let $e$ be a positive integer and define $\mathcal{C}_e = \mathcal{A}_e \cup \mathcal{B}_e$ where $\mathcal{A}_e$ is the set of all $x \in \mathcal{A}$ whose number of digits is less than $e$, and where $\mathcal{B}_e$ is the set of all integers whose number of digits is equal to $e$, which do not contain $0$, and whose number of occurrences $n'_j$ of digit $j$ is at most $n_j$ for $2 \leq j \leq 9$.\\
If $\mathcal{C}_e$ does not contain any integer divisible by $2^e$, then the same holds for $\mathcal{A}$.
\label{lemma-bound}
\end{lemma}

\begin{proof}
Let $x \in \mathcal{A}$. If the number of digits of $x$ is less than $e$, then $x \in \mathcal{A}_e$, and thus is not divisible by $2^e$. In the other case, $x_0 = x \bmod 10^e$ necessarily belongs to $\mathcal{B}_e$. As the sum of $x_0$ which is not divisible by $2^e$ and $(x-x_0)$ which is divisible by $10^e$, $x$ is thus not divisible by $2^e$.
\qed
\end{proof}

Thus, if we can find the smallest $e$ (if it exists) for which $\mathcal{B}_e$ does not contain any integer divisible by $2^e$, then $2^{e-1}$ is the maximal power of two in any $x \in \mathcal{A}$.\smallskip

As an example, consider $\mathcal{A} = \decall{0}{0}{2}{0}{0}{1}{0}{0}$ which corresponds to the multi-set $\{4,4,7\}$ of interest for $d=2$ and $s=112$. Lemma~\ref{lemma-bound} is verified for $e=8$ and not verified for any $e < 8$, thus $2^7$ is the maximal power of two in any $x \in \mathcal{A}$ (indeed $2^7$ divides 111744).\smallskip

Table~\ref{table-bound} gives the maximal power of two $2^a$ in any $x \in \mathcal{A}$ for all multi-sets that arise when still considering $d=2$ and $s=112$.\smallskip

\begin{table}
\begin{center}
\begin{tabular}{|c|c|c|}\hline
~~multi-set~~&~~$a$~~&\den{representative reaching divisibility by $2^a$}\\\hline\hline
~~(2,2,2,2,7)~~&~~13~~&\den{$172122112=21011\times 2^{13}$}\\\hline
~~(2,2,4,7)~~&~~15~~&\den{$211111411712=17\times 378977\times 2^{15}$}\\\hline
~~(4,4,7)~~&~~7~~&\den{$111744=3^2 \times 97 \times 2^7$}\\\hline
~~(4,7,8)~~&~~9~~&\den{$1178112=3\times 13\times 59\times 2^9$}\\\hline
\end{tabular}
\end{center}
\caption{Maximal powers of two arising for all multi-sets related to $d=2$ and $s=112$.}
\label{table-bound}
\end{table}

\section{In Conclusion}

Finally, solving the equations shows what we wanted: the graphs we gave, $B_1$, $B_3$, $B_5$, $B_7$ and $B_9$ are indeed the trees of pre-images of respectively $1$, $3$, $5$, $7$ and $9$. As seen above, this gives for $d\in \{1,3,5,7,9\}$ the form of all numbers $n$ such that $\Delta(n) = d$: those are the elements of

$$\bigcup\limits_{v=\tcs{\alpha}{\beta}{\gamma}\in V_d} \bigcup\limits_{\alpha_1 + 2 \alpha_2 = \alpha} \dec{\alpha_1}{\beta}{\gamma}{\alpha_2} $$

We also got the following optimal bounds:
\begin{itemize}
    \item If $\Delta(n)\in\{1,3,7,9\}$, then $\Xi(n) \leq 1$
    \item If $\Delta(n)=5$, then $\Xi(n) \leq 5$
\end{itemize}

It follows that the Multiplicative Persistence conjecture is proved for all odd targets and, in addition, $\Xi(2n+1)\leq5$ rather than $\Xi(n)\leq11$ in the general case.

\section{Further Research}

A natural question is the applicability of our strategy to even targets. Indeed, if successful, this will settle definitely the multiplicative persistence enigma. As is, the method that we just applied would require a prohibitive amount of calculations although according to our estimates, tackling those cases would be within the reach of Grover's algorithm on a quantum computer. Three other natural extension directions would be the simplification of the proofs provided in this paper (in case more elementary arguments could be used to reach the same results), the extension of our techniques to non-decimal bases as well as their generalization to the ``Erd\H{o}s-variant'' of the conjecture mentioned in \cite{rg}.


\bibliographystyle{alpha}
\bibliography{mers}

\begin{thebibliography}{McE19}

\bibitem[dFT14]{ref7}
Edson de~Faria and Charles Tresser.
\newblock On {S}loane's persistence problem.
\newblock {\em Experimental Mathematics}, 23(4):363--382, 2014.

\bibitem[Dia11]{ref6}
Mark~R. Diamond.
\newblock Multiplicative persistence base 10: some new null results.
\newblock {\em Online}, 2011.
\newblock \url{https://tinyurl.com/awv6m76e}.

\bibitem[K.81]{rg}
Guy~Richard K.
\newblock {\em Unsolved problems in number theory}.
\newblock Problem books in mathematics. Springer-Verlag, New York Berlin
  Heidelberg, 1981.

\bibitem[McE19]{ref5}
Kevin McElwee.
\newblock An algorithm for multiplicative persistence research.
\newblock {\em Online}, July 13 2019.
\newblock \url{https://tinyurl.com/4dvyx6jd}.

\bibitem[PS]{ref4}
Stephanie Perez and Robert Styer.
\newblock Persistence: A digit problem.
\newblock {\em Online}.
\newblock \url{https://tinyurl.com/5236zfvd}.

\bibitem[Sch]{ref2}
Walter Schneider.
\newblock The persistence of a number.
\newblock {\em Online}.
\newblock \url{https://tinyurl.com/t8mmckvp}.

\bibitem[Slo73]{sloane}
Neil J.~A. Sloane.
\newblock The persistence of a number.
\newblock {\em J. Recreational Mathematics}, 6:97--98, 1973.

\bibitem[Wei]{ref3}
Eric Weisstein.
\newblock World of mathematics, multiplicative persistence.
\newblock {\em Online}.
\newblock \url{https://tinyurl.com/hu9b3szw}.

\bibitem[Wor80]{ref1}
Susan Worst.
\newblock Multiplicative persistence of base four numbers.
\newblock {\em Scanned copy of manuscript and correspondence}, May 1980.
\newblock \url{https://tinyurl.com/33nspma4}.

\end{thebibliography}

\newpage
\appendix
\section{Appendix: Parameters for Equations}
\label{App1}

{\small
\begin{center}
    \begin{tabular}{|c|c||c|c|c|c|c|c|c|c|c|} \hline
$~~\Delta(n)~~$& $~~\mbox{\textsf{Eq ID}}~~$  &~~~$h$~~~&~~~$c_1$~~~&~~~$c_2$~~~&~~~$c_3$~~~& ~~~$c_4$~~~&~~~$c_5$~~~&~~~$c_6$~~~&~~~$c_7$~~~&~~~$\tau_h$~~~\\\hline
 
1&~~1.01~~  &~~~$0$~~~&~~~ ~~~&~~~ ~~~&~~~ ~~~& ~~~ ~~~&~~~ ~~~&~~~ ~~~&~~~ ~~~&~~~$-1$~~~\\\hline

3&~~3.01~~  &~~~$0$~~~&~~~1~~~&~~~ ~~~&~~~ ~~~& ~~~ ~~~&~~~ ~~~&~~~ ~~~&~~~ ~~~&~~~$-1$~~~\\\hline

7&~~7.01~~  &~~~$0$~~~&~~~3~~~&~~~ ~~~&~~~ ~~~& ~~~ ~~~&~~~ ~~~&~~~ ~~~&~~~ ~~~&~~~$-1$~~~\\\hline

9&~~9.01~~  &~~~$0$~~~&~~~1~~~&~~~1~~~&~~~ ~~~& ~~~ ~~~&~~~ ~~~&~~~ ~~~&~~~ ~~~&~~~$-1$~~~\\\hline

9&~~9.02~~  &~~~$0$~~~&~~~4~~~&~~~ ~~~&~~~ ~~~& ~~~ ~~~&~~~ ~~~&~~~ ~~~&~~~ ~~~&~~~$-1$~~~\\\hline

5&~~5.01~~  &~~~$1$~~~&~~~ ~~~&~~~ ~~~&~~~ ~~~& ~~~ ~~~&~~~ ~~~&~~~ ~~~&~~~ ~~~&~~~$7$~~~\\\hline

5&~~5.02~~  &~~~$1$~~~&~~~$1$~~~&~~~ ~~~&~~~ ~~~& ~~~ ~~~&~~~ ~~~&~~~ ~~~&~~~ ~~~&~~~$7$~~~\\\hline

5&~~5.03~~  &~~~$1$~~~&~~~$3$~~~&~~~ ~~~&~~~ ~~~& ~~~ ~~~&~~~ ~~~&~~~ ~~~&~~~ ~~~&~~~$7$~~~\\\hline

5&~~5.04~~  &~~~$2$~~~&~~~ ~~~&~~~ ~~~&~~~ ~~~& ~~~ ~~~&~~~ ~~~&~~~ ~~~&~~~ ~~~&~~~$23$~~~\\\hline

5&~~5.05~~  &~~~$1$~~~&~~~$1$~~~&~~~$2$~~~&~~~ ~~~& ~~~ ~~~&~~~ ~~~&~~~ ~~~&~~~ ~~~&~~~$7$~~~\\\hline

5&~~5.06~~  &~~~$1$~~~&~~~$1$~~~&~~~$1$~~~&~~~$1$~~~& ~~~ ~~~&~~~ ~~~&~~~ ~~~&~~~ ~~~&~~~$7$~~~\\\hline

5&~~5.07~~  &~~~$1$~~~&~~~$1$~~~&~~~$4$~~~&~~~ ~~~& ~~~ ~~~&~~~ ~~~&~~~ ~~~&~~~ ~~~&~~~$7$~~~\\\hline

5&~~5.08~~  &~~~$1$~~~&~~~$2$~~~&~~~$3$~~~&~~~ ~~~& ~~~ ~~~&~~~ ~~~&~~~ ~~~&~~~ ~~~&~~~$7$~~~\\\hline

5&~~5.09~~  &~~~$2$~~~&~~~$2$~~~&~~~ ~~~&~~~ ~~~& ~~~ ~~~&~~~ ~~~&~~~ ~~~&~~~ ~~~&~~~$23$~~~\\\hline 

5&~~5.10~~  &~~~$1$~~~&~~~$1$~~~&~~~$1$~~~&~~~$3$~~~& ~~~ ~~~&~~~ ~~~&~~~ ~~~&~~~ ~~~&~~~$7$~~~\\\hline 

5&~~5.11~~  &~~~$2$~~~&~~~$1$~~~&~~~$1$~~~&~~~~~~& ~~~ ~~~&~~~ ~~~&~~~ ~~~&~~~ ~~~&~~~$23$~~~\\\hline 

5&~~5.12~~  &~~~$3$~~~&~~~$1$~~~&~~~ ~~~&~~~~~~& ~~~ ~~~&~~~ ~~~&~~~ ~~~&~~~ ~~~&~~~$19$~~~\\\hline 

5&~~5.13~~  &~~~$1$~~~&~~~$3$~~~&~~~$4$~~~&~~~ ~~~& ~~~ ~~~&~~~ ~~~&~~~ ~~~&~~~ ~~~&~~~$7$~~~\\\hline 

5&~~5.14~~  &~~~$2$~~~&~~~$4$~~~&~~~ ~~~&~~~ ~~~& ~~~ ~~~&~~~ ~~~&~~~ ~~~&~~~ ~~~&~~~$23$~~~\\\hline 

5&~~5.15~~  &~~~$1$~~~&~~~$1$~~~&~~~$1$~~~&~~~$2$~~~& ~~~$3$~~~&~~~ ~~~&~~~ ~~~&~~~ ~~~&~~~$7$~~~\\\hline 

5&~~5.16~~  &~~~$2$~~~&~~~$1$~~~&~~~$1$~~~&~~~$2$~~~& ~~~ ~~~&~~~ ~~~&~~~ ~~~&~~~ ~~~&~~~$23$~~~\\\hline 

5&~~5.17~~  &~~~$3$~~~&~~~$1$~~~&~~~$2$~~~&~~~~~~& ~~~ ~~~&~~~ ~~~&~~~ ~~~&~~~ ~~~&~~~$19$~~~\\\hline 

5&~~5.18~~  &~~~$1$~~~&~~~$2$~~~&~~~$3$~~~&~~~$4$~~~& ~~~ ~~~&~~~ ~~~&~~~ ~~~&~~~ ~~~&~~~$7$~~~\\\hline 

5&~~5.19~~  &~~~$2$~~~&~~~$2$~~~&~~~$4$~~~&~~~~~~& ~~~ ~~~&~~~ ~~~&~~~ ~~~&~~~ ~~~&~~~$23$~~~\\\hline 

5&~~5.20~~  &~~~$1$~~~&~~~$3$~~~&~~~$3$~~~&~~~$3$~~~& ~~~ ~~~&~~~ ~~~&~~~ ~~~&~~~ ~~~&~~~$7$~~~\\\hline 

5&~~5.21~~  &~~~$2$~~~&~~~$3$~~~&~~~$3$~~~&~~~~~~& ~~~ ~~~&~~~ ~~~&~~~ ~~~&~~~ ~~~&~~~$23$~~~\\\hline 

5&~~5.22~~  &~~~$1$~~~&~~~$1$~~~&~~~$1$~~~&~~~$1$~~~& ~~~$2$~~~&~~~$2$~~~&~~~ ~~~&~~~ ~~~&~~~$7$~~~\\\hline 

5&~~5.23~~  &~~~$1$~~~&~~~$1$~~~&~~~$2$~~~&~~~$2$~~~& ~~~$4$~~~&~~~~~~&~~~ ~~~&~~~ ~~~&~~~$7$~~~\\\hline

5&~~5.24~~  &~~~$1$~~~&~~~$1$~~~&~~~$1$~~~&~~~$1$~~~& ~~~$1$~~~&~~~$1$~~~&~~~$3$~~~&~~~ $3$~~~&~~~$7$~~~\\\hline

5&~~5.25~~  &~~~$2$~~~&~~~$1$~~~&~~~$1$~~~&~~~$1$~~~& ~~~$1$~~~&~~~$1$~~~&~~~$3$~~~&~~~ ~~~&~~~$23$~~~\\\hline

5&~~5.26~~  &~~~$3$~~~&~~~$1$~~~&~~~$1$~~~&~~~$1$~~~& ~~~$1$~~~&~~~$3$~~~&~~~~~~&~~~ ~~~&~~~$19$~~~\\\hline

5&~~5.27~~  &~~~$1$~~~&~~~$1$~~~&~~~$1$~~~&~~~$1$~~~& ~~~$3$~~~&~~~$3$~~~&~~~$4$~~~&~~~ ~~~&~~~$7$~~~\\\hline

5&~~5.28~~  &~~~$2$~~~&~~~$1$~~~&~~~$1$~~~&~~~$1$~~~& ~~~$3$~~~&~~~$4$~~~&~~~~~~&~~~ ~~~&~~~$23$~~~\\\hline

5&~~5.29~~  &~~~$3$~~~&~~~$1$~~~&~~~$1$~~~&~~~$3$~~~& ~~~$4$~~~&~~~~~~&~~~~~~&~~~ ~~~&~~~$19$~~~\\\hline

5&~~5.30~~  &~~~$4$~~~&~~~$1$~~~&~~~$1$~~~&~~~$3$~~~& ~~~ ~~~&~~~ ~~~&~~~~~~&~~~ ~~~&~~~$119$~~~\\\hline

5&~~5.31~~  &~~~$1$~~~&~~~$1$~~~&~~~$3$~~~&~~~$3$~~~& ~~~$4$~~~&~~~$4$~~~&~~~~~~&~~~ ~~~&~~~$7$~~~\\\hline

5&~~5.32~~  &~~~$2$~~~&~~~$1$~~~&~~~$3$~~~&~~~$4$~~~& ~~~$4$~~~&~~~~~~&~~~~~~&~~~ ~~~&~~~$23$~~~\\\hline

5&~~5.33~~  &~~~$3$~~~&~~~$3$~~~&~~~$4$~~~&~~~$4$~~~& ~~~ ~~~&~~~~~~&~~~~~~&~~~ ~~~&~~~$19$~~~\\\hline

5&~~5.34~~  &~~~$4$~~~&~~~$3$~~~&~~~$4$~~~&~~~ ~~~& ~~~ ~~~&~~~~~~&~~~~~~&~~~ ~~~&~~~$119$~~~\\\hline

5&~~5.35~~  &~~~$1$~~~&~~~$1$~~~&~~~$1$~~~&~~~$2$~~~&~~~$3$~~~&~~~$3$~~~&~~~$3$~~~&~~~ ~~~&~~~$7$~~~\\\hline

5&~~5.36~~  &~~~$2$~~~&~~~$1$~~~&~~~$1$~~~&~~~$2$~~~&~~~$3$~~~&~~~$3$~~~&~~~~~~&~~~ ~~~&~~~$23$~~~\\\hline

5&~~5.37~~  &~~~$3$~~~&~~~$1$~~~&~~~$2$~~~&~~~$3$~~~&~~~$3$~~~&~~~~~~&~~~~~~&~~~ ~~~&~~~$19$~~~\\\hline

5&~~5.38~~  &~~~$1$~~~&~~~$2$~~~&~~~$3$~~~&~~~$3$~~~&~~~$3$~~~&~~~$4$~~~&~~~~~~&~~~ ~~~&~~~$7$~~~\\\hline

5&~~5.39~~  &~~~$2$~~~&~~~$2$~~~&~~~$3$~~~&~~~$3$~~~&~~~$4$~~~&~~~~~~&~~~~~~&~~~ ~~~&~~~$23$~~~\\\hline

\end{tabular}
\end{center}
}

\section{Appendix: Solutions}

For simplicity, we call \textsl{``Set of solutions of the equation of $E$ in $\mathbb{Z}$''} a set $\ssz{E}$ of tuples such that if $(\vec{a},u,w)$ is a solution in $\mathbb{Z}$ of $E$ verifying
\begin{center}
    \textit{(Property to avoid having to deal with equivalent solutions:) }\\
    
\[
\mbox{if~} (i<j \mbox{~and~} c_i = c_j) \mbox{~then~} (\exists (a'_i,a'_j) \in \llbracket 0; 11\rrbracket ^2)
    \mbox{~s.t.~} \left\{
    \begin{array}{l}
        a'_i \geq a'_j\mbox{~and}   \\
        (a_i,a_j) = (a'_i,a'_j) \bmod 12  \\
        \end{array}\right.
  \]

\end{center} then $\exists(u',w')\in \mathbb{Z}^2$ such that $(\vec{a},u',w') \in \ssz{E}$ and $u = u' \bmod \ord_{\dtcs{y}{0}{z}{0}}(3)$ and $w = w' \bmod \ord_{\dtcs{y}{0}{z}{0}}(7)$. Indeed, such an $\ssz{E}$ is what the algorithm returns: we do not look for the exact values of $u,w$ in $\mathbb{Z}$ since we do not need them, and we do not verify that all candidate tuples actually match a solution in $\mathbb{Z}$.

\subsection{Solving Equation 1.01}
The algorithm returns the following $\ssz{1.01}$; 

\begin{center}
\begin{tabular}{|c|c|c|}\hline
\textbf{~~element of $\ssz{3.01}$~~}   & \textbf{~~interpretation~~} & \textbf{~~conclusion~~}\\\hline\hline
\textcolor{blue}{((1), 2, 0)} &\textcolor{blue}{$f(x)=1$}&\textcolor{blue}{$\checkmark$}\\\hline
\end{tabular}
\end{center}

\subsection{Solving Equation 3.01}
The algorithm returns the following $\ssz{3.01}$:

\begin{center}
\begin{tabular}{|c|c|c|}\hline
\textbf{~~element of $\ssz{3.01}$~~}   & \textbf{~~interpretation~~} & \textbf{~~conclusion~~}\\\hline\hline
\textcolor{blue}{((1, 0), 3, 0)} &\textcolor{blue}{$f(x)=3$}&\textcolor{blue}{$\checkmark$}\\\hline
((1, 1), 3, 1) &~~\dnv~$a_0>a_1$~~& dismissed\\\hline
\end{tabular}
\end{center}

\subsection{Solving Equation 7.01}
The algorithm returns the following $\ssz{7.01}$:

\begin{center}
\begin{tabular}{|c|c|c|}\hline
\textbf{~~element of $\ssz{7.01}$~~}   & \textbf{~~interpretation~~} & \textbf{~~conclusion~~}\\\hline\hline
\textcolor{blue}{((1, 0), 2, 1)} &\textcolor{blue}{$f(x)=7$}&\textcolor{blue}{$\checkmark$}\\\hline
\end{tabular}
\end{center}

\subsection{Solving Equation 9.01}
The algorithm returns $\ssz{9.01}=\emptyset$.

\subsection{Solving Equation 9.02}
The algorithm returns the following $\ssz{9.02}$:

\begin{center}
\begin{tabular}{|c|c|c|}\hline
\textbf{~~element of $\ssz{9.02}$~~}   & \textbf{~~interpretation~~} & \textbf{~~conclusion~~}\\\hline\hline
\textcolor{blue}{((1, 0), 4, 0)} &$\textcolor{blue}{f(x)=9}$&\textcolor{blue}{$\checkmark$}\\\hline
((1, 1), 6, 0) &~~\dnv~$a_0>a_1$~~& dismissed\\\hline
\end{tabular}
\end{center}

\subsection{5.$xy$ Equations with no Solutions}
The algorithm returns $\ssz{5.xy}=\emptyset$ for the equations:
\textbf{
\begin{center}
 \begin{tabular}{cccccccccc}
~~5.13~~&~~5.14~~&~~5.16~~&~~5.17~~&~~5.25~~&~~5.26~~&~~5.30~~&~~5.33~~&~~5.34~~\\ 
\end{tabular}  
\end{center}}

\subsection{Solving Equation 5.01}
The algorithm returns the following $\ssz{5.01}$:

\begin{center}
\begin{tabular}{|c|c|c|}\hline
\textbf{~~element of $\ssz{5.01}$~~}   & \textbf{~~interpretation~~} & \textbf{~~conclusion~~}\\\hline\hline
\textcolor{blue}{((0), 2, 0)} & \textcolor{blue}{$f(x)=5$} &\textcolor{blue}{$\checkmark$}\\\hline
\textcolor{blue}{((1), 3, 0)} & \textcolor{blue}{$f(x)=15$} &\textcolor{blue}{$\checkmark$}\\\hline
\end{tabular}
\end{center}

$f(x)=5$ and $f(x)=15$ are thus the only $f(x)$ such that $f^2(x)=5$.

\subsection{Solving Equation 5.02}

The algorithm returns the following $\ssz{5.02}$:

\begin{center}
\begin{tabular}{|c|c|c|}\hline
\textbf{~~element of $\ssz{5.02}$~~}   & \textbf{~~interpretation~~} & \textbf{~~conclusion~~}\\\hline\hline
\textcolor{blue}{((1, 0), 2, 1)} & \textcolor{blue}{$f(x)=35$} &\textcolor{blue}{$\checkmark$}\\\hline
\textcolor{blue}{((2, 0), 5, 0)} & \textcolor{blue}{$f(x)=135$} &\textcolor{blue}{$\checkmark$}\\\hline
\textcolor{blue}{((2, 1), 4, 1)}& \textcolor{blue}{$f(x)=315$} &\textcolor{blue}{$\checkmark$}\\\hline
\end{tabular}
\end{center}

$35,135,315$, are thus the only $f(x)$ such that $f^2(x)=15$.

\subsection{Solving Equation 5.03}
The algorithm returns the following $\ssz{5.03}$:

\begin{center}
\begin{tabular}{|c|c|c|}\hline
\textbf{~~element of $\ssz{5.03}$~~}   & \textbf{~~interpretation~~} & \textbf{~~conclusion~~}\\\hline\hline
\textcolor{blue}{((3, 1), 2, 3)} & \textcolor{blue}{$f(x)=1715$} &\textcolor{blue}{$\checkmark$}\\\hline
\end{tabular}
\end{center}

\subsection{Solving Equation 5.04}
The algorithm returns the following $\ssz{5.04}$:

\begin{center}
\begin{tabular}{|c|c|c|}\hline
\textbf{~~element of $\ssz{5.04}$~~}   & \textbf{~~interpretation~~} & \textbf{~~conclusion~~}\\\hline\hline
\textcolor{blue}{((0), 3, 0)} & \textcolor{blue}{$f(x)=75$} &\textcolor{blue}{$\checkmark$}\\\hline
\textcolor{blue}{((1), 2, 1)} & \textcolor{blue}{$f(x)=175$} &\textcolor{blue}{$\checkmark$}\\\hline
\end{tabular}
\end{center}

These two solutions, together with solution of Equation 5.03 give us that, if $f^2(x)=35$, then $f(x)\in\{75,175,1715\}$.

\subsection{Solving Equation 5.05}
The algorithm returns the following $\ssz{5.05}$:

\begin{center}
\begin{tabular}{|c|c|c|}\hline
\textbf{~~element of $\ssz{5.05}$~~}   & \textbf{~~interpretation~~} & \textbf{~~conclusion~~}\\\hline\hline
((0, 1, 0), 2, 2) &~~\dnv~$a_0>a_1,a_2\geq 0$ with $a_1 \neq a_2$~~& dismissed\\\hline
((3, 1, 1), 2, 3) &~~\dnv~$a_0>a_1,a_2\geq 0$ with $a_1 \neq a_2$~~& dismissed\\\hline
\end{tabular}
\end{center}

\subsection{Solving Equation 5.06}
The algorithm returns the following $\ssz{5.06}$:

\begin{center}
\noindent\makebox[\textwidth]{
\begin{tabular}{|c|c|c|}\hline
\textbf{~~element of $\ssz{5.06}$~~}   & \textbf{~~interpretation~~} & \textbf{~~conclusion~~}\\\hline\hline
((0, 1, 0, 0), 2, 2) &~~$a_1,a_2,a_3$ \anad~~& dismissed\\\hline
((3, 1, 1, 1), 2, 3) &~~$a_1,a_2,a_3$ \anad~~& dismissed\\\hline
\end{tabular}}
\end{center}

\subsection{Solving Equations 5.07 and 5.08}
The algorithm coincidentally returns the following $\ssz{5.07}=\ssz{5.08}$:

\begin{center}
\noindent\makebox[\textwidth]{
\begin{tabular}{|c|c|c|}\hline
\textbf{~~element of $\ssz{5.07}$ and $\ssz{5.08}$~~}   & \textbf{~~interpretation~~} & \textbf{~~conclusion~~}\\\hline\hline
((0, 0, 0), 3, 1) &~~\dnv~$a_1\neq a_2$~~& dismissed\\\hline
((3, 0, 0), 7, 0) &~~\dnv~$a_1\neq a_2$~~& dismissed\\\hline
\end{tabular}}
\end{center}

\subsection{Solving Equation 5.09}
The algorithm returns the following $\ssz{5.09}$:

\begin{center}
\begin{tabular}{|c|c|c|}\hline
\textbf{~~element of $\ssz{5.09}$~~}   & \textbf{~~interpretation~~} & \textbf{~~conclusion~~}\\\hline\hline
\textcolor{blue}{((2, 0), 4, 1)} & \textcolor{blue}{$f(x)=1575$} &\textcolor{blue}{$\checkmark$}\\\hline
\end{tabular}
\end{center}

According to solutions of Equations 5.08 and 5.09, 1575 is the only $f(x)$ such that $f^2(x)=175$.

\subsection{Solving Equation 5.10}
The algorithm returns the following $\ssz{5.10}$:

\begin{center}
\noindent\makebox[\textwidth]{
\begin{tabular}{|c|c|c|}\hline
\textbf{~~element of $\ssz{5.10}$~~}   & \textbf{~~interpretation~~} & \textbf{~~conclusion~~}\\\hline\hline
((0, 0, 0, 0), 3, 1) &~~$a_1,a_2,a_3$ \anad~~& dismissed\\\hline
((3, 0, 0, 0), 7, 0) &~~$a_1,a_2,a_3$ \anad~~& dismissed\\\hline
\end{tabular}}
\end{center}

\subsection{Solving Equation 5.11}
The algorithm returns the following $\ssz{5.11}$:

\begin{center}
\noindent\makebox[\textwidth]{
\begin{tabular}{|c|c|c|}\hline
\textbf{~~element of $\ssz{5.11}$~~}   & \textbf{~~interpretation~~} & \textbf{~~conclusion~~}\\\hline\hline
((2, 0, 0), 4, 1) &~~\dnv~$a_1\neq a_2$~~& dismissed\\\hline
\end{tabular}}
\end{center}

\subsection{Solving Equation 5.12}
The algorithm returns the following $\ssz{5.12}$:

\begin{center}
\begin{tabular}{|c|c|c|}\hline
\textbf{~~element of $\ssz{5.12}$~~}   & \textbf{~~interpretation~~} & \textbf{~~conclusion~~}\\\hline\hline
\textcolor{blue}{((1, 0), 5, 0)} & \textcolor{blue}{$f(x)=3375$} &\textcolor{blue}{$\checkmark$}\\\hline
\end{tabular}
\end{center}

According to solutions of Equations 5.10 to 5.14, 3375 is the only $f(x)$ such that $f^2(x)=315$.

\subsection{Solving Equation 5.15}
The algorithm returns the following $\ssz{5.15}$:

\begin{center}
\noindent\makebox[\textwidth]{
\begin{tabular}{|c|c|c|}\hline
\textbf{~~element of $\ssz{5.15}$~~}   & \textbf{~~interpretation~~} & \textbf{~~conclusion~~}\\\hline\hline
((1, 3, 0, 2, 3), 2, 5) &~~ $a_1,\ldots,a_4$ \anad~~& dismissed\\\hline
\end{tabular}}
\end{center}

\subsection{Solving Equation 5.18}
The algorithm returns the following $\ssz{5.18}$:

\begin{center}
\noindent\makebox[\textwidth]{
\begin{tabular}{|c|c|c|}\hline
\textbf{~~element of $\ssz{5.18}$~~}   & \textbf{~~interpretation~~} & \textbf{~~conclusion~~}\\\hline\hline
((1, 0, 1, 0), 3, 2) &~~$a_1,a_2,a_3$ \anad~~& dismissed\\\hline
((4, 1, 2, 2), 8, 1) &~~$a_1,a_2,a_3$ \anad~~& dismissed\\\hline
\end{tabular}}
\end{center}

\subsection{Solving Equation 5.19}
The algorithm returns the following $\ssz{5.19}$:

\begin{center}
\noindent\makebox[\textwidth]{
\begin{tabular}{|c|c|c|}\hline
\textbf{~~element of $\ssz{5.19}$~~}   & \textbf{~~interpretation~~} & \textbf{~~conclusion~~}\\\hline\hline
((1, 0, 1,) 2, 3) &~~\dnv~$a_0>a_1,a_2$~~& dismissed\\\hline
\end{tabular}}
\end{center}

\subsection{Solving Equation 5.20}
The algorithm returns the following $\ssz{5.20}$:

\begin{center}
\noindent\makebox[\textwidth]{
\begin{tabular}{|c|c|c|}\hline
\textbf{~~element of $\ssz{5.20}$~~}   & \textbf{~~interpretation~~} & \textbf{~~conclusion~~}\\\hline\hline
((1, 1, 0, 0), 3, 2) &~~$a_1,a_2,a_3$ \anad~~& dismissed\\\hline
\end{tabular}}
\end{center}

\subsection{Solving Equation 5.21}
The algorithm returns the following $\ssz{5.21}$:

\begin{center}
\begin{tabular}{|c|c|c|}\hline
\textbf{~~element of $\ssz{5.21}$~~}   & \textbf{~~interpretation~~} & \textbf{~~conclusion~~}\\\hline\hline
\textcolor{blue}{((3, 2, 1), 4, 3)} & \textcolor{blue}{$f(x)=77175$} &\textcolor{blue}{$\checkmark$}\\\hline
\end{tabular}
\end{center}

According to solutions of Equations 5.20 and 5.21, 77175 is the only $f(x)$ such that $f^2(x)=1715$.

\subsection{Solving Equation 5.22}
The algorithm returns the following $\ssz{5.22}$:

\begin{center}
\noindent\makebox[\textwidth]{
\begin{tabular}{|c|c|c|}\hline
\textbf{~~element of $\ssz{5.22}$~~}   & \textbf{~~interpretation~~} & \textbf{~~conclusion~~}\\\hline\hline
((1, 2, 2, 0, 3, 3), 2, 5) &~~$a_1,\ldots,a_5$ \anad~~& dismissed\\\hline
((1, 3, 3, 0, 3, 2), 2, 5) &~~$a_1,\ldots,a_5$ \anad~~& dismissed\\\hline
\end{tabular}}
\end{center}

\subsection{Solving Equation 5.23}
The algorithm returns the following $\ssz{5.23}$:

\begin{center}
\noindent\makebox[\textwidth]{
\begin{tabular}{|c|c|c|}\hline
\textbf{~~element of $\ssz{5.23}$~~}   & \textbf{~~interpretation~~} & \textbf{~~conclusion~~}\\\hline\hline
((1, 1, 1, 0, 0), 3, 2) &~~$a_1,\ldots,a_4$ \anad~~& dismissed\\\hline
((1, 3, 2, 2, 2), 3, 4) &~~$a_1,\ldots,a_4$ \anad~~& dismissed\\\hline
\textcolor{blue}{((4, 0, 3, 1, 2), 7, 2)} &~~\textcolor{blue}{$f(x)=59535$}~~&\textcolor{blue}{$\checkmark$}\\\hline
((4, 1, 2, 0, 0), 4, 3) &~~$a_1,\ldots,a_4$ \anad~~& dismissed\\\hline
((4, 2, 2, 1, 2), 8, 1) &~~$a_1,\ldots,a_4$ \anad~~& dismissed\\\hline
\end{tabular}}
\end{center}

According to solutions of Equations 5.22 and 5.23, 59535 is the only $f(x)$ such that $f^2(x)=3375$.

\subsection{Solving Equation 5.24}
The algorithm returns the following $\ssz{5.24}$:

\begin{center}
\noindent\makebox[\textwidth]{
\begin{tabular}{|c|c|c|}\hline
\textbf{~~element of $\ssz{5.24}$~~}   & \textbf{~~interpretation~~} & \textbf{~~conclusion~~}\\\hline\hline
((0, 1, 0, 0, 0, 0, 0, 0), \phantom{1}6, 0) &~~$a_1,\ldots,a_7$ \anad~~& dismissed\\\hline
((0, 1, 0, 0, 0, 0, 1, 0), \phantom{1}5, 1) &~~$a_1,\ldots,a_7$ \anad~~& dismissed\\\hline
((0, 1, 1, 1, 1, 0, 0, 0), \phantom{1}5, 1) &~~$a_1,\ldots,a_7$ \anad~~& dismissed\\\hline
((0, 1, 1, 1, 1, 1, 1, 1), \phantom{1}4, 2) &~~$a_1,\ldots,a_7$ \anad~~& dismissed\\\hline
((0, 2, 0, 0, 0, 0, 0, 0), \phantom{1}4, 2) &~~$a_1,\ldots,a_7$ \anad~~& dismissed\\\hline
((0, 2, 1, 1, 0, 0, 1, 1), \phantom{1}8, 0) &~~$a_1,\ldots,a_7$ \anad~~& dismissed\\\hline
((0, 2, 1, 1, 0, 0, 2, 0), \phantom{1}7, 1) &~~$a_1,\ldots,a_7$ \anad~~& dismissed\\\hline
((0, 3, 1, 1, 1, 1, 2, 2), 10, 0) &~~$a_1,\ldots,a_7$ \anad~~& dismissed\\\hline
((0, 3, 2, 2, 2, 1, 2, 1), 10, 0) &~~$a_1,\ldots,a_7$ \anad~~& dismissed\\\hline
((0, 3, 3, 1, 0, 0, 2, 0), \phantom{1}5, 3) &~~$a_1,\ldots,a_7$ \anad~~& dismissed\\\hline
((0, 3, 3, 2, 0, 0, 3, 2), \phantom{1}4, 4) &~~$a_1,\ldots,a_7$ \anad~~& dismissed\\\hline
((0, 4, 0, 0, 0, 0, 5, 0), 13, 1) &~~$a_1,\ldots,a_7$ \anad~~& dismissed\\\hline
((0, 5, 5, 5, 4, 0, 0, 0), 13, 1) &~~$a_1,\ldots,a_7$ \anad~~& dismissed\\\hline
\end{tabular}}
\end{center}

\subsection{Solving Equation 5.27}
The algorithm returns the following $\ssz{5.27}$:

\begin{center}
\noindent\makebox[\textwidth]{
\begin{tabular}{|c|c|c|}\hline
\textbf{~~element of $\ssz{5.27}$~~}   & \textbf{~~interpretation~~} & \textbf{~~conclusion~~}\\\hline\hline
((1, 1, 0, 0, 1, 0, 1), 2, 3) &~~$a_1,\ldots,a_6$ \anad~~& dismissed\\\hline
((1, 1, 1, 0, 1, 1, 0), 2, 3) &~~$a_1,\ldots,a_6$ \anad~~& dismissed\\\hline
((2, 1, 1, 0, 0, 0, 0), 3, 2) &~~$a_1,\ldots,a_6$ \anad~~& dismissed\\\hline
((4, 5, 5, 3, 2, 1, 3), 2, 7) &~~$a_1,\ldots,a_6$ \anad~~& dismissed\\\hline
((5, 3, 1, 1, 2, 1, 1), 6, 3) &~~$a_1,\ldots,a_6$ \anad~~& dismissed\\\hline
((5, 4, 1, 0, 2, 2, 1), 5, 4) &~~$a_1,\ldots,a_6$ \anad~~& dismissed\\\hline
((5, 5, 0, 0, 2, 0, 4), 7, 4) &~~$a_1,\ldots,a_6$ \anad~~& dismissed\\\hline
((5, 5, 4, 0, 4, 2, 0), 7, 4) &~~$a_1,\ldots,a_6$ \anad~~& dismissed\\\hline
\end{tabular}}
\end{center}





\subsection{Solving Equation 5.28}
The algorithm returns the following $\ssz{5.28}$:

\begin{center}
\noindent\makebox[\textwidth]{
\begin{tabular}{|c|c|c|}\hline
\textbf{~~element of $\ssz{5.28}$~~}   & \textbf{~~interpretation~~} & \textbf{~~conclusion~~}\\\hline\hline
((2, 0, 0, 0, 1, 0), 2, 3) &~~$a_1,\ldots,a_5$ \anad~~& dismissed\\\hline
((2, 1, 1, 1, 0, 0), 2, 3) &~~$a_1,\ldots,a_5$ \anad~~& dismissed\\\hline
((4, 1, 0, 0, 1, 1), 8, 1) &~~$a_1,\ldots,a_5$ \anad~~& dismissed\\\hline
((4, 2, 2, 2, 4, 2), 8, 3) &~~$a_1,\ldots,a_5$ \anad~~& dismissed\\\hline
((4, 4, 4, 4, 2, 2), 8, 3) &~~$a_1,\ldots,a_5$ \anad~~& dismissed\\\hline
\end{tabular}}
\end{center}

\subsection{Solving Equation 5.29}
The algorithm returns the following $\ssz{5.29}$:

\begin{center}
\noindent\makebox[\textwidth]{
\begin{tabular}{|c|c|c|}\hline
\textbf{~~element of $\ssz{5.29}$~~}   & \textbf{~~interpretation~~} & \textbf{~~conclusion~~}\\\hline\hline
((0, 0, 0, 0, 0), 3, 2) &~~$a_1,\ldots,a_4$ \anad~~& dismissed\\\hline
((0, 1, 1, 1, 2), 3, 4) &~~$a_1,\ldots,a_4$ \anad~~& dismissed\\\hline
((0, 2, 1, 2, 1), 3, 4) &~~$a_1,\ldots,a_4$ \anad~~& dismissed\\\hline
((3, 1, 0, 1, 1), 9, 0) &~~$a_1,\ldots,a_4$ \anad~~& dismissed\\\hline
((3, 1, 1, 0, 0), 5, 2) &~~$a_1,\ldots,a_4$ \anad~~& dismissed\\\hline
\end{tabular}}
\end{center}

\subsection{Solving Equation 5.31}
The algorithm returns the following $\ssz{5.31}$:

\begin{center}
\noindent\makebox[\textwidth]{
\begin{tabular}{|c|c|c|}\hline
\textbf{~~element of $\ssz{5.31}$~~}   & \textbf{~~interpretation~~} & \textbf{~~conclusion~~}\\\hline\hline
((0, 2, 1, 1, 2, 1), 2, 4) &~~$a_1,\ldots,a_5$ \anad~~& dismissed\\\hline
((1, 0, 0, 0, 0, 0), 4, 1) &~~$a_1,\ldots,a_5$ \anad~~& dismissed\\\hline
((1, 0, 1, 1, 1, 1), 6, 1) &~~$a_1,\ldots,a_5$ \anad~~& dismissed\\\hline
((1, 0, 2, 1, 2, 1), 4, 3) &~~$a_1,\ldots,a_5$ \anad~~& dismissed\\\hline
((1, 0, 5, 2, 4, 1), 6, 5) &~~\dnv~ $a_0>a_1,\ldots,a_5\geq 0$~~& dismissed\\\hline
((1, 1, 0, 0, 1, 0), 7, 0) &~~$a_1,\ldots,a_5$ \anad~~& dismissed\\\hline
((1, 2, 0, 0, 2, 1), 9, 0) &~~$a_1,\ldots,a_5$ \anad~~& dismissed\\\hline
((1, 2, 1, 0, 0, 0), 6, 1) &~~$a_1,\ldots,a_5$ \anad~~& dismissed\\\hline
((3, 2, 0, 0, 3, 1), 2, 5) &~~$a_1,\ldots,a_5$ \anad~~& dismissed\\\hline
\end{tabular}}
\end{center}

\subsection{Solving Equation 5.32}
The algorithm returns the following $\ssz{5.32}$:

\begin{center}
\noindent\makebox[\textwidth]{
\begin{tabular}{|c|c|c|}\hline
\textbf{~~element of $\ssz{5.32}$~~}   & \textbf{~~interpretation~~} & \textbf{~~conclusion~~}\\\hline\hline
((0, 2, 0, 3, 3), 5, 4) &~~$a_1,\ldots,a_4$ \anad~~& dismissed\\\hline
((1, 3, 2, 2, 2), 2, 5) &~~$a_1,\ldots,a_4$ \anad~~& dismissed\\\hline
\end{tabular}}
\end{center}

\subsection{Solving Equation 5.35}
The algorithm returns the following $\ssz{5.35}$:

\begin{center}
\noindent\makebox[\textwidth]{
\begin{tabular}{|c|c|c|}\hline
\textbf{~~element of $\ssz{5.35}$~~}   & \textbf{~~interpretation~~} & \textbf{~~conclusion~~}\\\hline\hline
((1, 0, 0, 1, 1, 1, 0), 2, 3) &~~$a_1,\ldots,a_6$ \anad~~& dismissed\\\hline
((1, 1, 1, 0, 1, 1, 0), 2, 3) &~~$a_1,\ldots,a_6$ \anad~~& dismissed\\\hline
((2, 0, 0, 1, 0, 0, 0), 3, 2) &~~$a_1,\ldots,a_6$ \anad~~& dismissed\\\hline
((2, 1, 1, 0, 0, 0, 0), 3, 2) &~~$a_1,\ldots,a_6$ \anad~~& dismissed\\\hline
((4, 3, 3, 5, 3, 2, 1), 2, 7) &~~$a_1,\ldots,a_6$ \anad~~& dismissed\\\hline
((4, 5, 5, 3, 3, 2, 1), 2, 7) &~~$a_1,\ldots,a_6$ \anad~~& dismissed\\\hline
((5, 2, 2, 3, 4, 1, 1), 4, 5) &~~$a_1,\ldots,a_6$ \anad~~& dismissed\\\hline
((5, 3, 1, 1, 2, 1, 1), 6, 3) &~~$a_1,\ldots,a_6$ \anad~~& dismissed\\\hline
((5, 3, 2, 2, 1, 1, 1), 6, 3) &~~$a_1,\ldots,a_6$ \anad~~& dismissed\\\hline
((5, 3, 3, 2, 4, 1, 1), 4, 5) &~~$a_1,\ldots,a_6$ \anad~~& dismissed\\\hline
((5, 4, 0, 1, 2, 2, 1), 5, 4) &~~$a_1,\ldots,a_6$ \anad~~& dismissed\\\hline
((5, 5, 4, 0, 4, 2, 0), 7, 4) &~~$a_1,\ldots,a_6$ \anad~~& dismissed\\\hline
\end{tabular}}
\end{center}






\subsection{Solving Equation 5.36}

The algorithm returns the following $\ssz{5.36}$:

\begin{center}
\noindent\makebox[\textwidth]{
\begin{tabular}{|c|c|c|}\hline
\textbf{~~element of $\ssz{5.36}$~~}   & \textbf{~~interpretation~~} & \textbf{~~conclusion~~}\\\hline\hline
((2, 0, 0, 0, 1, 0), 2, 3) &~~$a_1,\ldots,a_5$ \anad~~& dismissed\\\hline
((2, 1, 0, 1, 0, 0), 2, 3) &~~$a_1,\ldots,a_5$ \anad~~& dismissed\\\hline
((4, 0, 0, 1, 1, 1), 8, 1) &~~$a_1,\ldots,a_5$ \anad~~& dismissed\\\hline
((4, 1, 1, 0, 1, 1), 8, 1) &~~$a_1,\ldots,a_5$ \anad~~& dismissed\\\hline
((4, 2, 2, 2, 4, 2), 8, 3) &~~$a_1,\ldots,a_5$ \anad~~& dismissed\\\hline
((4, 4, 2, 4, 2, 2), 8, 3) &~~$a_1,\ldots,a_5$ \anad~~& dismissed\\\hline
\end{tabular}}
\end{center}

\subsection{Solving Equation 5.37}

The algorithm returns the following $\ssz{5.37}$:

\begin{center}
\noindent\makebox[\textwidth]{
\begin{tabular}{|c|c|c|}\hline
\textbf{~~element of $\ssz{5.37}$~~}   & \textbf{~~interpretation~~} & \textbf{~~conclusion~~}\\\hline\hline
((0, 0, 0, 0, 0), 3, 2) &~~$a_1,\ldots,a_4$ \anad~~& dismissed\\\hline
((0, 2, 1, 2, 1), 3, 4) &~~$a_1,\ldots,a_4$ \anad~~& dismissed\\\hline
((3, 0, 1, 0, 0), 5, 2) &~~$a_1,\ldots,a_4$ \anad~~& dismissed\\\hline
((3, 0, 1, 1, 1), 9, 0) &~~$a_1,\ldots,a_4$ \anad~~& dismissed\\\hline
\end{tabular}}
\end{center}

\subsection{Solving Equation 5.38}

The algorithm returns the following $\ssz{5.38}$:

\begin{center}
\noindent\makebox[\textwidth]{
\begin{tabular}{|c|c|c|}\hline
\textbf{~~element of $\ssz{5.38}$~~}   & \textbf{~~interpretation~~} & \textbf{~~conclusion~~}\\\hline\hline
((0, 2, 2, 1, 1, 1), \phantom{1}2, 4) &~~$a_1,\ldots,a_5$ \anad~~& dismissed\\\hline
((1, 0, 0, 0, 0, 0), \phantom{1}4, 1) &~~$a_1,\ldots,a_5$ \anad~~& dismissed\\\hline
((1, 1, 1, 0, 0, 0), \phantom{1}7, 0) &~~$a_1,\ldots,a_5$ \anad~~& dismissed\\\hline
((1, 1, 2, 0, 0, 0), \phantom{1}5, 2) &~~$a_1,\ldots,a_5$ \anad~~& dismissed\\\hline
((1, 1, 2, 2, 2, 3), 11, 0) &~~$a_1,\ldots,a_5$ \anad~~& dismissed\\\hline
((1, 2, 1, 1, 1, 1), \phantom{1}5, 2) &~~$a_1,\ldots,a_5$ \anad~~& dismissed\\\hline
((1, 2, 2, 0, 0, 1), \phantom{1}9, 0) &~~$a_1,\ldots,a_5$ \anad~~& dismissed\\\hline
\end{tabular}}
\end{center}

\subsection{Solving Equation 5.39}
The algorithm returns the following $\ssz{5.39}$:

\begin{center}
\noindent\makebox[\textwidth]{
\begin{tabular}{|c|c|c|}\hline
\textbf{~~element of $\ssz{5.39}$~~}   & \textbf{~~interpretation~~} & \textbf{~~conclusion~~}\\\hline\hline
((0, 1, 0, 0, 0), \phantom{1}7, 0) &~~$a_1,\ldots,a_4$ \anad~~& dismissed\\\hline
((0, 2, 1, 0, 1), \phantom{1}9, 0) &~~$a_1,\ldots,a_4$ \anad~~& dismissed\\\hline
((0, 3, 1, 1, 2), 11, 0) &~~$a_1,\ldots,a_4$ \anad~~& dismissed\\\hline
((1, 3, 1, 1, 1), \phantom{1}2, 5) &~~$a_1,\ldots,a_4$ \anad~~& dismissed\\\hline
\end{tabular}}
\end{center}

The full Python code of the solving algorithm is available from the authors.

\section{Appendix: Graphs $B_2$ and $B_4$}
\label{app-U2-U4}
$B_2=(U_2,F_2)$ =\\
\begin{tikzpicture}
\tikzstyle{level 1}=[sibling distance=8.5em,level distance = 6em, edge from parent/.style = {draw, -latex},every node/.style       = {font=\footnotesize},rotate=90,minimum width=3em]
  \tikzstyle{level 2}=[sibling distance=5.6em,level distance = 6em, edge from parent/.style = {draw, -latex},every node/.style       = {font=\footnotesize},minimum width=3.5em]
    \tikzstyle{level 3}=[sibling distance=2.6em,level distance = 7em, edge from parent/.style = {draw, -latex},every node/.style       = {font=\footnotesize},minimum width=5.9em]
      \tikzstyle{level 4}=[sibling distance=1.7em,level distance = 8em, edge from parent/.style = {draw, -latex},every node/.style       = {font=\footnotesize},minimum width=5em]
            \tikzstyle{level 5}=[sibling distance=2em,level distance = 8em, edge from parent/.style = {draw, -latex},every node/.style       = {font=\footnotesize}]
  [
    grow                    = right,
    level distance          = 10em,
    edge from parent/.style = {draw, -latex},
    every node/.style       = {font=\footnotesize},
    sloped
  ]
  \node [node] {2}
    child { node [node] {112}
          child { node [node] {1728} 
        child { node [node] {11239424} 
        }
        child { node [node] {321489} 
        }
        child { node [node] {314928} 
        }
        child { node [node] {268912} 
        }
      }
      child { node [node] {2187} 
      }
    }
    child { node [node] {21}
    }
    child { node [node] {12}
      child { node [node] {13122} 
      }
      child { node [node] {216} 
        child { node [node] {1229312} 
        }
        child { node [node] {61236} 
          child { node [node] {\tiny{$\dtcs{16}{1}{0}{2}$}}}
        }
        child { node [node] {33614} 
        }
        child { node [node] {14336} 
        }
      }
      child { node [node] {162} 
        child { node [node] {93312} 
          child { node [node] {\tiny{$\dtcs{26}{3}{0}{0}$}} }
          child { node [node] {\tiny{$\dtcs{8}{3}{0}{5}$}} }
        }
      }
        child { node [node] {1134} 
      }
      child { node [node] {126} 
        child { node [node] {1792} 
          child { node [node] {7112448} }
          child { node [node] {1741824}
            child { node [node] {\tiny{$\dtcs{6}{0}{0}{8}$}
            } }
          }
          child { node [node] {1411788} }
          child { node [node] {8748} }
        }
          child { node [node] {729} 
        }
        child { node [node] {972} 
          child { node [node] {3111696} }
          child { node [node] {393216} }
        }
      }
    }

    edge [loop above] ()
    ;
\end{tikzpicture}

$B_4=(U_4,F_4)$ =\\
\begin{tikzpicture}
\tikzstyle{level 1}=[sibling distance=8.5em,level distance = 6em, edge from parent/.style = {draw, -latex},every node/.style       = {font=\footnotesize},rotate=90]
  \tikzstyle{level 2}=[sibling distance=5.6em,level distance = 6em, edge from parent/.style = {draw, -latex},every node/.style       = {font=\footnotesize}]
    \tikzstyle{level 3}=[sibling distance=2.8em,level distance = 7em, edge from parent/.style = {draw, -latex},every node/.style       = {font=\footnotesize},minimum width=5em]
  [
    grow                    = right,
    sibling distance        = 4em,
    level distance          = 6em,
    edge from parent/.style = {draw, -latex},
    every node/.style       = {font=\footnotesize},
    sloped
  ]
  \node [node] {4}
    child { node [node] {14}
      child { node [node] {72} 
          child { node [node] {1161216} 
              child { node [node] {$\dtcs{23}{7}{0}{1}$} }
          }
          child { node [node] {294} }
          child { node [node] {189} }
          child { node [node] {98} }
      }
      child { node [node] {27} }
    }
    edge [loop above] ()
    ;
\end{tikzpicture}

\end{document}